\newcommand{\TITLE}{Parallel Newton methods for the continuous
quadratic knapsack problem: A Jacobi and Gauss-Seidel tale}

\newcommand{\FUNDING}{This work was supported by
	CEPID-CeMEAI (FAPESP 2013/07375-0),
	FAPESP (grant numbers 2023/08706-1
	and
	2024/12967-8), and
	CNPq (grant numbers
    312394/ 2023-3
	and
	302520/2025-2).
}

\newcommand{\ABSTRACT}{The continuous quadratic knapsack (CQK) problem
involves minimizing a diagonal convex quadratic function subject to box
constraints and a single linear equality constraint. It has numerous
applications in resource allocation, multicommodity flow, machine learning,
and classical optimization tasks such as Lagrangian relaxation and
quasi-Newton updates. In this work, we revisit the semismooth Newton method
introduced by Cominetti, Mascarenhas, and Silva. We demonstrate that the
method can be significantly improved in two directions. First, for projections
onto the simplex or the $\ell_1$-ball, it can incorporate Condat's highly
effective initial multiplier guess. Second, it can serve as a flexible
foundation for CQK algorithms, allowing for different parallel variants
tailored to exploit CPU and GPU computational models. These improvements are
implemented in the open-source Julia package \texttt{NewtonCQK.jl}. We present
extensive numerical tests comparing this implementation with other
state-of-the-art solvers, demonstrating its superior efficiency and
scalability.}

\documentclass[a4paper]{article}

\usepackage[utf8]{inputenc}
\usepackage[T1]{fontenc}    % Allow search and copy in PDF files
\usepackage[english]{babel}
\usepackage{amsfonts}
\usepackage{amsmath}
\usepackage{amsthm}
\usepackage{amssymb}
\usepackage{a4}
\usepackage{enumerate}
\usepackage{graphicx}
\usepackage{setspace}
\usepackage{xcolor}
\usepackage{tabularx}
\usepackage{lipsum}
\usepackage{subfigure}
\usepackage{algorithm}
\usepackage[noend]{algpseudocode}
\usepackage{lscape}
\usepackage{pgfplots}
\usepackage{url}
\usepackage[colorlinks=true, citecolor=blue]{hyperref}

\usepackage[margin=3cm]{geometry}
\usepackage{longtable}

\usepackage[natbib=true, style=apa, backend=biber,sorting=nyt]{biblatex}
\let\cite\citep

\newtheorem{theorem}{Theorem}[section]

\newtheorem{lemma}{Lemma}[section]
\newtheorem{remark}{Remark}[section]

\title{\TITLE}

\author{
L. D. Secchin
\thanks{Department of Applied Mathematics, Federal University of Esp\'{i}rito
Santo, ES, Brazil. Email: {\tt leonardo.secchin@ufes.br}}
\and
P. J. S. Silva
\thanks{Department of Applied Mathematics, Universidade Estadual de Campinas,
Campinas, SP, Brazil. Email: {\tt pjssilva@ime.unicamp.br}
}
}

\newcommand{\R}{\mathbb{R}}
\newcommand{\sgn}{{\normalfont \textrm{sgn}\,}}

\graphicspath{{figures/}}

\begin{document}
\selectlanguage{english}

\maketitle
\begin{abstract}
\ABSTRACT
\end{abstract}

\section{Introduction}

The nonlinear knapsack problem is a classical model in operations research and
finance;
see~\citet{Pang1980,Bitran1981,Amaral2006,Bretthauer2002,Patriksson2008,Sharkey2011}
and references therein. In its most general form, it can be stated as
\begin{equation*}
	\min_x \ f(x) \quad \text{s.t.}\quad g(x) \leq b,\ x \in S,
\end{equation*}
where $f, g: \R^n \mapsto \R$, $b$ is a real constant, and $S$ is a nonempty
closed set. That is, it is a nonlinear optimization problem with a single
inequality constraint and abstract constraints. When the functions are
separable and the abstract constraint $S$ represents a box, the model is also
known as the continuous nonlinear resource allocation
problem~\cite{Patriksson2008}. Such models find applications in hierarchical
planning systems~\cite{Bitran1981}, portfolio selection, resource
allocation~\cite{Bretthauer1995,Bretthauer1997}, video on-demand
services~\cite{Patriksson2008}, batching~\cite{Patriksson2008}, capacity and
production planning~\cite{Bretthauer2002}, multicommodity
flows~\cite{Ventura1991}, and support vector
machines~\cite{Dai2005,Gonzalez2011,Cominetti2014}, among many others.

In this paper, we consider the particular case known as the \emph{Continuous
Quadratic Knapsack problem} (CQK)~\cite{Dai2005,Cominetti2014}, defined as
\begin{equation}
    \min_x \ \frac{1}{2} x^t Dx -a^t x \quad
    \text{s.t.} \quad
    b^tx = r, \quad l\leq x\leq u,
\tag{CQK}\label{CQK}
\end{equation}
where $x \in \R^n$, $D = \text{diag}(d_1, \ldots, d_n)$ is a positive diagonal
matrix, $l \leq u$ are the lower and upper vector bounds on $x$, respectively.
Some bounds may be $\pm \infty$. We assume throughout the text that $b > 0$.
This condition can be satisfied by eliminating variables or reverting their
sign where necessary. Since $D$ is positive definite, \eqref{CQK} is
equivalent to the projection of $D^{-1}a$ onto the feasible set with respect
to the norm induced by $D$. It encompasses the problem of projecting a point
$y$ onto the $r$-simplex:
\begin{equation}\label{PDelta}\tag{proj$_\Delta$}
	\min\ \frac{1}{2} \| x - y \|^2 \quad
	\text{s.t.}\quad x \in \Delta =
	\left\{x \in \R^n\: \bigg|\: \sum_{i = 1}^n x_i = r,\ x \geq 0 \right\}.
\end{equation}
Indeed, \eqref{PDelta} is, up to a constant in the objective function, the
special case of \eqref{CQK} with $D=I$, $a=y$, $b_i=1$, $l_i=0$, and
$u_i=\infty$ for all $i$. Furthermore, this problem is intrinsically linked
to the projection of $y$ onto the $\ell_1$-ball of radius $r$
\begin{equation*}
  \mathcal B = \left\{x\in \R^n \: \bigg| \: \sum_{i=1}^n |x_i| \leq r \right\},
\end{equation*}
which can be reduced to the projection of the vector of absolute values of
$y$ onto the $r$-simplex, subsequently restoring the original
signs~\cite{Duchi2008}. Note that~\eqref{CQK} and its variants appear
naturally when solving a continuous version of the general knapsack problem
using methods based on projected gradient methods, when the functional
constraint is linear. They also appear in the subproblems of algorithms that
linearize a nonlinear constraint, such as sequential quadratic
programming~\cite{Nocedal2006}. In turn, the continuous knapsack problems
appear as relaxations used in branch-and-bound methods needed when integrality
is required~\cite{Bretthauer2002}.

Another more recent source of applications is first-order methods for various
large-scale problems, such as the training of support vector
machines~\cite{Dai2005}, sparse recovery with a spectral projected gradient
solver~\cite{vandenBerg2008}, the computation of Wasserstein
barycenters~\cite{MAZO2024}, and continuous resource allocation problems with
instances having tens of millions of variables~\cite{Patriksson2015}. In the
first three, projections onto the feasible set of~\eqref{CQK}, onto $\mathcal
B$, or onto simplices are computed at each iteration. Projections onto these
sets have also recently become building blocks of machine learning models that
are typically trained on GPUs. For example, the sparsemax transformation, an
alternative to the softmax in neural attention mechanisms, is precisely the
projection onto $\Delta$ with $r=1$~\cite{Martins2016}, and GPU kernels for
its generalization, $\alpha$-entmax, have recently been
developed~\cite{Goncalves2025}. Another example is the computation of sparse
adversarial perturbations, which requires projections onto the intersection of
an $\ell_1$-ball and a box~\cite{Croce2021}, a problem that is closely related
to~\eqref{CQK}. In most of these cases, the projection must be computed
repeatedly, making its efficiency essential. Naturally, it is also important
to have methods that can exploit the parallel power of modern computer
architectures, with many cores and GPU accelerators.

Many methods to solve \eqref{CQK}, and especially to project onto $\Delta$ or
$\mathcal B$, have emerged in the literature. The most prominent ones employ
variable-fixing techniques, originally proposed in~\citet{Michelot1986}. In
particular, to the best of our knowledge, \citet{Condat2016} introduced the
fastest variable-fixing (sequential) algorithm for projecting onto $\Delta$.

Later, \citet{Dai2024} introduced a general framework to parallelize methods
based on variable-fixing. They proved that if a variable is fixed at a bound
within a subproblem, formed by a subset of the coordinates, it remains fixed
at the same bound in the original problem. This property allows for the
estimation of most of the fixed variables by splitting the full problem into
smaller subproblems that are solved in parallel. After this initial phase, the
original problem is solved in a reduced form with a significant portion of its
variables already fixed. This technique was used successfully to parallelize
Condat's method, resulting in a very efficient implementation.

Another approach, introduced in~\citet{Cominetti2014}, uses the fact
that~\eqref{CQK} is a strictly convex problem with simple bounds. Hence, it
can be solved through its dual. The authors proposed a globally convergent
semismooth Newton method to solve the dual problem, leveraging its simple
structure. The method is also very efficient and can be combined with
variable-fixing strategies.

In this paper, we present improvements to the semismooth Newton method with
optional variable fixing. Our objective is to obtain efficient parallel
methods for both multicore CPUs and GPU accelerators, combining ideas from the
two approaches above when appropriate. First, we specialize the semismooth
Newton method to project onto $\Delta$ and show that its initial multiplier
can be computed by adapting the Gauss-Seidel initialization used by Condat,
thereby improving performance. We also demonstrate that the resulting method
is directly parallelizable without using Dai and Chen's approach. This is
possible because our algorithm employs the dual function to drive convergence
instead of relying on variable-fixing for that purpose. Hence, the evaluation
of the dual function is shared among the parallel workers, and no computation
is discarded. In contrast, Dai and Chen's approach uses the solutions of the
subproblems only to identify fixed variables, and the number of variables
fixed in this phase may decrease as the number of processors increases and the
subproblems become smaller. Moreover, while variable-fixing works well on
CPUs, it is not easily ported to GPUs, which also hinders a GPU version of Dai
and Chen's approach. This suggests an alternative GPU variation of the
semismooth Newton algorithm that entirely removes variable-fixing, resulting
in a method with a fully decoupled iteration, similar to the Jacobi method for
linear systems.

We also discuss several implementation choices that have a significant impact
on performance and could also benefit the original implementations, such as
the data structures used to represent the free variables and strategies to
avoid redundant passes over the problem data. For example, while the
implementations of Condat and of Dai and Chen maintain the values $y_i$ of the
free variables throughout the iterations, we track their indices instead. This
avoids a full scan of the problem data to build the final solution.

Such developments lead to a new high-performance Julia implementation that is
openly available and can be installed with Julia's standard package manager.
It includes sequential, parallel, and GPU implementations of the semismooth
Newton method, together with specialized variants to project onto $\Delta$ and
$\mathcal B$. The sequential algorithm for~\eqref{PDelta} generally
outperforms even the original C implementation of Condat's method.
Furthermore, the parallel Newton code outperforms the parallel version of
Condat's algorithm by Dai and Chen. Finally, the GPU implementation is shown
to be highly effective for large-scale problems.

The paper is organized as follows. In Section~\ref{sec:literature}, we recall
the sequential Newton and variable-fixing methods for \eqref{CQK} and
important particular cases, notably for \eqref{PDelta}. In
Section~\ref{sec:newtonsimplex}, we present our specialized Newton method for
\eqref{PDelta}. In Section~\ref{sec:parallelNewton}, we explain how these
methods can be parallelized. Section~\ref{sec:numerical} is devoted to
numerical tests, showing that our proposal outperforms state-of-the-art
methods in the literature. Finally, Section~\ref{sec:conclusions} presents our
conclusions and outlines directions for future work.

\paragraph{Notation} $P_C(z)$ denotes the Euclidean projection of $z$ onto a
given convex set $C$. $|I|$ denotes the cardinality of a finite set $I$. We
define $\sgn(t)$ as $1$ if $t > 0$, $-1$ if $t < 0$, and $0$ if $t = 0$.

\section{Literature review on sequential methods for CQK}
\label{sec:literature}

In this section, we review the two primary frameworks for solving~\eqref{CQK}.
The first is based on Newtonian methods. A secant variation was introduced
in~\citet{Dai2005}, which was later extended into a fully semismooth Newton
method by~\citet{Cominetti2014}. Our presentation is based on the latter. The
second approach relies on variable fixing techniques, see~\citet{Michelot1986,
Kiwiel2008, Kiwiel2008a, Condat2016}. We focus on Condat's variant, which
incorporates a Gauss-Seidel acceleration that is highly effective in practice \cite{Condat2016}.

\subsection{Newton methods for CQK}
\label{sec:newtonCQK}

Dualizing only the linear constraint $b^tx=r$ in \eqref{CQK}, we obtain the
dual problem
\[
\max_{\lambda \in \R} F(\lambda), \quad F(\lambda) :=
	\min_{l \leq x \leq u} \frac{1}{2} x^t D x - a^t x + \lambda(r - b^t x).
\]
One can then easily show that the vector $x(\lambda)\in \R^n$ that attains the
minimum in the definition of the dual objective is given by,
\begin{equation}\label{xopt}
  \forall i = 1, \ldots, n,\quad x(\lambda)_i
  = P_{[l_i,u_i]}\big((b_i\lambda + a_i)/d_i\big)
  = \max\{l_i,\min\{u_i, (b_i\lambda + a_i)/d_i\}\}.
\end{equation}
Using basic duality results, primal-dual optimality is achieved when this
vector is primal feasible. That is, the condition
\begin{equation}\label{opt}
	\varphi(\lambda) = r, \quad
	\varphi(\lambda) := b^t x(\lambda)
	= \sum_{i=1}^n b_i \max\{l_i,\min\{u_i, (b_i\lambda + a_i)/d_i\}\}
\end{equation}
is equivalent to the optimality of $\lambda$ for the dual problem and
$x(\lambda)$ for \eqref{CQK}. Therefore, \eqref{CQK} is equivalent to solving
the nonlinear equation
\[
\varphi(\lambda) = r.
\]

The main idea in~\citet{Cominetti2014} is to apply a (semismooth) Newtonian
strategy to this equation. Since the functions $\lambda \to
P_{[l_i,u_i]}\big((b_i\lambda + a_i)/d_i\big)$ are piecewise linear, so is
$\varphi$ and its graph possibly changes the slope exactly at the
\emph{breakpoints}
\begin{equation*}
  \underline l_i = (d_il_i-a_i)/{b_i}
  \quad \text{and} \quad
  \overline u_i = (d_iu_i-a_i)/{b_i}.
\end{equation*}
Also, observe that the lateral derivatives of $\varphi$ are non-negative and
given by
\begin{equation}\label{phideriv}
  \varphi_+'(\lambda) =
  \sum_{i\,\mid \, \underline l_i\leq \lambda < \overline u_i} b_i^2/d_i
  \quad\text{and}\quad
  \varphi_-'(\lambda) =
  \sum_{i\,\mid \, \underline l_i < \lambda \leq \overline u_i} b_i^2/d_i,
\end{equation}
see Figure~\ref{fig:varphi}(a). The function $\varphi$ is non-decreasing and
may be constant between consecutive breakpoints as depicted in
Figure~\ref{fig:varphi}(b). Consequently, given $\lambda_k$, the Newtonian
step depends on whether $\varphi(\lambda_k) < r$ or $\varphi(\lambda_k) > r$.
In the first case, $\lambda$ must be increased. Thus $\varphi'_+(\lambda_k)$
is used, if not zero. Otherwise, $\lambda_k$ must be decreased, and
$\varphi'_-(\lambda_k)$ is employed. If the appropriate lateral derivative is
zero, the closest breakpoint in the required direction is computed instead.
The complete method is presented in Algorithm~\ref{alg:purenewtonCQK}. Note
that it also incorporates a secant update to avoid poor Newton steps, an idea
inspired by the pure secant method suggested in~\citet{Fletcher2006}.

\begin{figure}[ht]
  \centering
  \subfigure[]{\def\svgwidth{170px}
  %% Creator: Inkscape 1.2.2 (b0a8486541, 2022-12-01), www.inkscape.org
%% Accompanies image file 'phi_deriv_svg-tex.pdf' (pdf, eps, ps)
\begingroup%
  \makeatletter%
  \providecommand\color[2][]{%
    \errmessage{(Inkscape) Color is used for the text in Inkscape, but the package 'color.sty' is not loaded}%
    \renewcommand\color[2][]{}%
  }%
  \providecommand\transparent[1]{%
    \errmessage{(Inkscape) Transparency is used (non-zero) for the text in Inkscape, but the package 'transparent.sty' is not loaded}%
    \renewcommand\transparent[1]{}%
  }%
  \providecommand\rotatebox[2]{#2}%
  \newcommand*\fsize{\dimexpr\f@size pt\relax}%
  \newcommand*\lineheight[1]{\fontsize{\fsize}{#1\fsize}\selectfont}%
  \ifx\svgwidth\undefined%
    \setlength{\unitlength}{185.73113162bp}%
    \ifx\svgscale\undefined%
      \relax%
    \else%
      \setlength{\unitlength}{\unitlength * \real{\svgscale}}%
    \fi%
  \else%
    \setlength{\unitlength}{\svgwidth}%
  \fi%
  \global\let\svgwidth\undefined%
  \global\let\svgscale\undefined%
  \makeatother%
  \begin{picture}(1,0.48284445)%
    \lineheight{1}%
    \setlength\tabcolsep{0pt}%
    \put(0,0){\includegraphics[width=\unitlength,page=1]{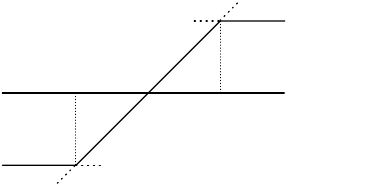}}%
    \put(0.55658796,0.20052613){\color[rgb]{0,0,0}\makebox(0,0)[lt]{\lineheight{1.25}\smash{\begin{tabular}[t]{l}$\overline u_i$\end{tabular}}}}%
    \put(0.18476851,0.26241031){\color[rgb]{0,0,0}\makebox(0,0)[lt]{\lineheight{1.25}\smash{\begin{tabular}[t]{l}$\underline l_i$\end{tabular}}}}%
    \put(0.27933748,0.04829528){\color[rgb]{0,0,0}\makebox(0,0)[lt]{\lineheight{1.25}\smash{\begin{tabular}[t]{l}$\varphi'_-(\underline l_i)$\end{tabular}}}}%
    \put(-0.00197811,0.00765719){\color[rgb]{0,0,0}\makebox(0,0)[lt]{\lineheight{1.25}\smash{\begin{tabular}[t]{l}$\varphi'_+(\underline l_i)$\end{tabular}}}}%
    \put(0.62530848,0.45217383){\color[rgb]{0,0,0}\makebox(0,0)[lt]{\lineheight{1.25}\smash{\begin{tabular}[t]{l}$\varphi'_-(\overline u_i)$\end{tabular}}}}%
    \put(0.33863063,0.41849487){\color[rgb]{0,0,0}\makebox(0,0)[lt]{\lineheight{1.25}\smash{\begin{tabular}[t]{l}$\varphi'_+(\overline u_i)$\end{tabular}}}}%
  \end{picture}%
\endgroup%
}
  \quad
  \subfigure[]{\def\svgwidth{170px}
  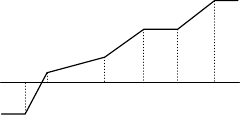}
  \caption{(a) Lateral slopes of $\max\{l_i,\min\{u_i, (b_i\lambda +
    a_i)/d_i\}\}$ at breakpoints. (b) A typical plot of $\varphi(\cdot)-r$.}
  \label{fig:varphi}
\end{figure}
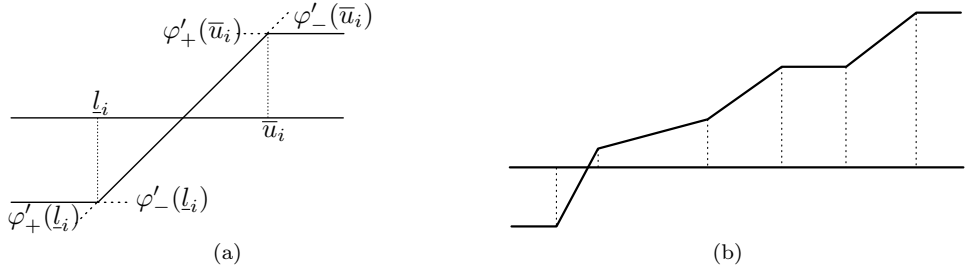

\begin{algorithm}[ht]
	\caption{Semismooth Newton method for solving
    \eqref{CQK}}\label{alg:purenewtonCQK}
  \begin{algorithmic}[1]
    \Require the initial dual estimate $\lambda_0$
    \State Set $k\leftarrow 0$, $\underline\lambda \leftarrow -\infty$ and $\overline\lambda\leftarrow \infty$
    \If {$\varphi(\lambda_k)=r$}\label{alg:purenewtonCQKstop}
    \State Stop and return $x(\lambda_k)$ as the solution of \eqref{CQK}
    \EndIf
    \If {$\varphi(\lambda_k) < r$}
      \State Update $\underline\lambda\leftarrow \lambda_k$
      \If {$\varphi'_+(\lambda_k)>0$}
        \State Define $\tilde \lambda_k = \lambda_k - \frac{\varphi(\lambda_k)-r}{\varphi'_+(\lambda_k)}$\label{alg:purenewtonCQKposstep}
        \State If $\tilde \lambda_k<\overline\lambda$, define $\lambda_{k+1} = \tilde \lambda_k$ and go to line \ref{alg:purenewtonCQKend}. Otherwise, go to line \ref{alg:purenewtonCQKsecant}

      \Else
        \State Find the closest breakpoint to the right of $\underline\lambda$, that is, find the smallest $\underline l_i$ such that $\underline l_i > \underline\lambda$.
        If such a breakpoint exists, define $\lambda_{k+1} = \underline l_i$ and go to line \ref{alg:purenewtonCQKend}. Otherwise, stop declaring \eqref{CQK} infeasible \label{alg:purenewtonCQKendbreakright}
      \EndIf
    \EndIf

    \If {$\varphi(\lambda_k) > r$}
      \State Update $\overline\lambda\leftarrow \lambda_k$
      \If {$\varphi'_-(\lambda_k)>0$}
        \State Define $\tilde \lambda_k = \lambda_k - \frac{\varphi(\lambda_k)-r}{\varphi'_-(\lambda_k)}$\label{alg:purenewtonCQKnegstep}
        \State If $\tilde \lambda_k>\underline\lambda$, define $\lambda_{k+1} = \tilde \lambda_k$ and go to line \ref{alg:purenewtonCQKend}. Otherwise, go to line \ref{alg:purenewtonCQKsecant}

      \Else
        \State Find the closest breakpoint to the left of $\overline\lambda$, that is, find the largest $\overline u_i$ such that $\overline u_i < \overline\lambda$.
        If such a breakpoint exists, define $\lambda_{k+1} = \overline u_i$ and go to line \ref{alg:purenewtonCQKend}. Otherwise, stop declaring \eqref{CQK} infeasible \label{alg:purenewtonCQKendbreakleft}
      \EndIf
    \EndIf

    \State Define $\lambda_{k+1}$ as the secant step
    $
    \lambda_S = \underline\lambda - \frac{(\varphi(\underline\lambda) - r)
    (\overline\lambda - \underline\lambda)}{\varphi(\overline\lambda) -
    \varphi(\underline\lambda)} \in (\underline\lambda,\overline\lambda).
    $
    If there are breakpoints in $(\underline\lambda,\overline\lambda)$, project $\lambda_{k+1}$ onto the interval between the closest breakpoints to the right of $\underline\lambda$ and to the left of $\overline\lambda$ \label{alg:purenewtonCQKsecant}

    \State Set $k\leftarrow k+1$ and go to line \ref{alg:purenewtonCQKstop} \label{alg:purenewtonCQKend}
  \end{algorithmic}
\end{algorithm}

The secant step (line \ref{alg:purenewtonCQKsecant}) is used to avoid
cycling~\cite{Cominetti2014}. It computes the zero of the affine function
whose graph interpolates $(\underline\lambda, \varphi(\underline\lambda) - r)$ and
$(\overline\lambda, \varphi(\overline\lambda) - r)$, where the interval
$[\underline\lambda,\overline\lambda]$ contains the optimal $\lambda^*$, hence
the name ``secant''. The resulting $\lambda_S$ belongs to
$(\underline\lambda,\overline\lambda)$ in exact arithmetic. The projection onto the interval
between the closest breakpoints guarantees that at least one breakpoint leaves
the interval $(\underline\lambda,\overline\lambda)$ at the next iteration. This
approach is slightly simpler than the secant step in Algorithm~2
of \citet{Cominetti2014}, since it does not require the computation of
$\varphi(\lambda_S)$. This strategy is sufficient for the global convergence of
Algorithm~\ref{alg:purenewtonCQK} which, according to~Theorem 2
of~\citet{Cominetti2014}, converges in at most $4n+1$ iterations with an
overall complexity of $O(n^2)$ arithmetic operations.

\begin{remark}
    The secant step $\lambda_S$ may suffer from numerical instabilities when
    $\varphi(\overline\lambda)\approx\varphi(\underline\lambda)$. In such
    cases, it may lie outside the interval $(\underline{\lambda},
    \overline{\lambda})$. When this occurs, our implementation sets
    $\lambda_{k+1}\leftarrow (\underline{\lambda} + \overline{\lambda})/2$,
    thereby guaranteeing that $\lambda_{k+1}\in(\underline{\lambda},
    \overline{\lambda})$. Note that if this safeguard is applied repeatedly,
    the method reduces to the slowly convergent bisection method.

    On the other hand, the Newton steps in lines~\ref{alg:purenewtonCQKposstep} and~\ref{alg:purenewtonCQKnegstep} of
    Algorithm~\ref{alg:purenewtonCQK} may become unstable when $d_i\approx0$
    for all $i$, due to the expressions for $\varphi'_+(\lambda)$ and
    $\varphi'_-(\lambda)$ (see~\ref{phideriv}). This may trigger secant steps
    and thus slow convergence on highly ill-conditioned problems. To avoid
    this, the user may rescale the data to avoid small values in $d$. It is
    worth mentioning that this phenomenon was not observed in our tests.
\end{remark}

As suggested in \citet{Cominetti2014}, the initial dual variable $\lambda_0$
can be defined as
\begin{equation}\label{CQKlambda0}
	\lambda_0 = \frac{r-s}{q},
	\quad s = \sum_{i\in J} \frac{b_ia_i}{d_i},
	\quad q = \sum_{i\in J} \frac{b_i^2}{d_i},
\end{equation}
with $J=\{1, 2, \ldots, n\}$. This is the dual multiplier corresponding to the
solution ignoring all variable bounds, i.e., $x_i(\lambda) = (b_i\lambda +
a_i)/d_i$ for all $i$. It resembles the initialization in the variable fixing
method of \citet{Kiwiel2008a}. Alternatively, if a primal estimate $\bar x$ is
provided, we assume that it conveys information about the bounds that are
expected to be binding at the solution: variables reaching or violating their
bounds in $\bar x$ are expected to have active bounds at the solution. Hence,
$\lambda_0$ is computed with $J = \{i \mid l_i < \bar x_i < u_i\}$. If this
set is empty, we discard $\bar x$ and set $J=\{1, 2, \ldots, n\}$.

\subsection{Variable fixing methods}

As stated in \citet{Cominetti2014}, the Newton method can be combined with a
variable fixing technique because $\varphi$ is non-decreasing. In fact, $x_i$
can be fixed to $l_i$ if $\varphi(\lambda_k) > r$ and $x(\lambda_k)_i =
\underline l_i$. On the other hand, it can be fixed to $u_i$ if
$\varphi(\lambda_k) < r$ and $x(\lambda_k)_i = \overline u_i$. In the first
case, the optimal $\lambda_* \leq \lambda_k$, while in the second, $\lambda_*
\geq \lambda_k$. This can reduce the number of operations in
Algorithm~\ref{alg:purenewtonCQK}, as fixed variables can be neglected in the
computation of $\varphi$ and its derivatives. It was verified in
\citet{Cominetti2014} that this can be effective in speeding up the method. In
our implementation, we attempt to fix as many variables as possible before the
secant step (line \ref{alg:purenewtonCQKsecant} of
Algorithm~\ref{alg:purenewtonCQK}).

This principle forms the basis of the variable fixing algorithms developed
in \citet{Condat2016, Kiwiel2008, Kiwiel2008a}, and references therein. These
methods iteratively identify variables that are fixed at their bounds in the
solution until the optimal set of free variables is determined, allowing the
final solution to be computed using its associated multiplier estimate. In
this work, we focus on Condat's method for projection onto a simplex, as it
is widely considered the most efficient variant available.

\subsubsection{Projecting onto a simplex}

Let us focus on solving~\eqref{PDelta}. Variable fixing methods aim to
identify the subset of free variables at the optimum. Given a tentative set of
free variables $J \subseteq \{1, 2, \ldots, n \}$, a candidate primal-dual pair
can be computed using specialized versions of~\eqref{CQKlambda0}
and~\eqref{xopt}~\cite{Condat2016}:
\begin{equation}\label{lambdaJ}
	\lambda_J = \frac{r-\sum_{i\in J} y_i}{|J|},\quad
  x^J_i = \max\{0, y_i + \lambda_J \},\ \forall i = 1, \ldots, n.
\end{equation}
Now, recalling that an optimal dual solution $\lambda_*$ recovers the optimal
primal solution $x^*$ through~\eqref{xopt}, it follows that
\begin{displaymath}
\sum_{i = 1}^n \max\{0, y_i + \lambda_* \} =
\sum_{i=1}^n x^*_i = r =
\sum_{i\in J} (y_i + \lambda_J) \leq
\sum_{i=1}^n \max\{0, y_i + \lambda_J \} = \sum_{i = 1}^n x^J_i.
\end{displaymath}
Consequently, $\lambda_J \geq \lambda_*$.

This observation is the fundamental building block of variable fixing methods.
The process begins by assuming that all variables are potentially free at the
solution, i.e., $J=\{1,\ldots,n\}$. The algorithm then verifies whether the
associated $x^J$ is feasible; if so, the problem is solved. Otherwise, one or
more variables that must be fixed at $0$, the lower bound, are identified, and
the set $J$ is reduced before starting a new
iteration~\cite{Kiwiel2008, Kiwiel2008a}.

\subsubsection{Condat's method to project onto simplices and $\ell_1$ balls}

\citet{Condat2016} made a key observation that resulted in the most widely
used algorithm for projection onto a simplex. He realized that as soon as a
variable is identified as fixed, the set $J$ and its respective multiplier
estimate $\lambda_J$ can be updated immediately within the same iteration,
introducing a Gauss-Seidel flavor to the method. This approach leads directly
to Algorithms~\ref{alg:lambda0simplex} and~\ref{alg:condat}.

Algorithm~\ref{alg:lambda0simplex} constructs an initial candidate subset of
free variables by adding one variable at a time, provided it cannot be proven
that the variable must be fixed at $0$ in the optimal solution.
Algorithm~\ref{alg:condat} then refines this set by identifying further
variables to fix. In both cases, the multiplier estimate is constantly updated
in a Gauss-Seidel fashion.

\begin{algorithm}[ht]
	\caption{Initialize $\lambda$ (simplex projection)}
	\label{alg:lambda0simplex}
	\begin{algorithmic}[1]
		\Require a subset of variable indices $I = \{i_1,\ldots,i_p\}\subseteq \{1,\ldots,n\}$.
		\State Set $J\leftarrow \{i_1\}$, $\widetilde J\leftarrow \emptyset$, $\lambda\leftarrow r-y_{i_1}$
		\For {$j = 2,\ldots, p$}
      \If {$y_{i_j}+\lambda > 0$}\label{alg:lambda0simplexlambdaif1}
        \State Set $\lambda \leftarrow (r - \sum_{\ell\in J}y_{i_\ell} - y_{i_j})/(|J|+1)$\label{alg:lambda0simplexlambda1}
        \If {$\lambda < r-y_{i_j}$}
        \State Update $J\leftarrow J\cup \{i_j\}$
        \Else
        \State Set $\widetilde J\leftarrow \widetilde J\cup J$, $J\leftarrow \{i_j\}$ and $\lambda \leftarrow r - y_{i_j}$\label{alg:lambda0simplexlambda3}
        \EndIf
      \EndIf
		\EndFor
		\ForAll{$j\in \widetilde J$}
      \If {$y_{j}+\lambda > 0$}\label{alg:lambda0simplexlambdaif2}
      \State Set $\lambda \leftarrow (r - \sum_{\ell\in J}y_{i_\ell} - y_j)/(|J|+1)$ and $J\leftarrow J\cup \{j\}$\label{alg:lambda0simplexlambda2}
      \EndIf
		\EndFor
		\State \Return
		the list of non-fixed variables $J\subseteq I$
	\end{algorithmic}
\end{algorithm}

\begin{algorithm}[ht]
	\caption{Condat's method to project a point onto a simplex}
	\label{alg:condat}
	\begin{algorithmic}[1]
		\Require the point $y$ to be projected
		\State Compute $\lambda$ and $J\subseteq \{1,\ldots,n\}$ by Algorithm \ref{alg:lambda0simplex} with $I=\{1,\ldots,n\}$
		\While {$J$ changes}
      \ForAll {$j\in J$}
        \If {$y_{j}+\lambda \leq 0$}
        \State Set $J\leftarrow J\backslash \{j\}$ and update $\lambda \leftarrow \lambda + (y_j+\lambda)/|J|$\label{alg:condatuplambda}
        \EndIf
      \EndFor
		\EndWhile
		\State \Return the solution $x^*$ of \eqref{PDelta} defined by $x^*_i=\max\{0,y_i+\lambda\}$
	\end{algorithmic}
\end{algorithm}

Condat's method (and any other method for projecting onto a simplex) can be used
to compute projections onto an $\ell_1$ ball. In fact, if $y\not\in \mathcal B$,
the projection of $y$ onto $\mathcal B$ can be recovered from projecting the
entry-wise absolute value $|y| = (|y_1|,\ldots,|y_n|)$ onto $\Delta$ and
multiplying its result by the original signs of $y$, see Proposition~2.1 of
\citet{Condat2016}. Clearly, if $y\in \mathcal B$, the projection of $y$ onto
$\mathcal B$ is $y$ itself.

In this case, we can also derive a more aggressive variable fixing result. The
next lemma, whose proof follows from Proposition~2.1 of \citet{Condat2016} and
\eqref{xopt}--\eqref{phideriv}, establishes that the $i$-th coordinate of the
optimal projection $x_i^*$ is zero whenever $y_i = 0$ and $y \not\in \mathcal
B$. This allows us to specialize the variable fixing criterion used in the
computation of $\lambda_0$ beyond that of Algorithm \ref{alg:lambda0simplex}.
Specifically, the conditions $y_{i_j} + \lambda > 0$ and $y_{j} + \lambda > 0$
in lines~\ref{alg:lambda0simplexlambdaif1} and \ref{alg:lambda0simplexlambdaif2}
are changed to $\min\{|y_{i_j}|,|y_{i_j}| + \lambda\} > 0$ and $\min\{|y_{j}|,|y_{j}| +
\lambda\} > 0$, respectively. Clearly, $y_l$ must be changed to $|y_l|$
elsewhere in the logic.

\begin{lemma}\label{lem:l1ballfixvar} Suppose that $\sum_{i=1}^n |y_i|\geq r$
	and let $x^* = P_{\mathcal B}(y)$. Then, for all $i = 1, \ldots, n$, $x_i^*
	= \sgn(y_i) x(\lambda_*)_i = \sgn(y_i) \max\{0,|y_i|+\lambda_*\}$, where
	$\lambda_*\leq 0$ solves the dual of \eqref{PDelta} with $|y|$ in place of $y$.
\end{lemma}

\section{A specialized Newton method to project onto a simplex}
\label{sec:newtonsimplex}

We now discuss the specialization of Algorithm~\ref{alg:purenewtonCQK} for
solving~\eqref{PDelta}, with the objective of deriving a method that is
competitive with Condat's algorithm. This problem is equivalent
to~\eqref{CQK}, up to a constant in the objective, by setting $D = I$, $a = y$ and
$r > 0$, while taking $b_i = 1$, $l_i = 0$ and $u_i = +\infty$ for all $i$. The
expressions~\eqref{xopt}--\eqref{phideriv} simplify accordingly. In the
following, we demonstrate that with appropriate initialization,
Algorithm~\ref{alg:purenewtonCQK} avoids vanishing derivatives. This property
allows for a streamlined implementation, resulting in
Algorithm~\ref{alg:purenewtonsimplex}.

\begin{theorem}\label{lem:Newtonproj} Consider
  Algorithm~\ref{alg:purenewtonCQK} applied to \eqref{PDelta}. Suppose that
  $\lambda_0\geq \underline y=\min_i\{-y_i\}$ is not a zero of
  $\varphi(\lambda) - r$. Then
  \begin{enumerate}
	\item $\varphi'_+(\lambda_0)>0$ and, if $\varphi(\lambda_0) > r$,
	$\varphi'_-(\lambda_0)>0$;
	\item For all $k \geq 0,\ \lambda_{k+1} = \tilde\lambda_k\in [\underline
	\lambda, \overline \lambda]$;
	\item For all $k \geq 1,\ \lambda_k\geq \lambda_{k+1} \geq \lambda_* >
    \underline y$, where $\lambda_*$ is the zero of $\varphi(\lambda) - r$;
	\item For all $k \geq 1,\ \varphi(\lambda_k) > r$ and
    $\varphi'_-(\lambda_k)>0$.
  \end{enumerate}
\end{theorem}

\begin{proof}
The only breakpoints in this case are $\underline l_i = -y_i$. From
\eqref{opt} and \eqref{phideriv}, we have $\varphi(\lambda) = \sum_{i=1}^n
\max\{0, y_i+\lambda\}$, $\varphi_+'(\lambda) = \sum_{i\,\mid \, -y_i\leq
\lambda} 1$ and $\varphi_-'(\lambda) = \sum_{i\,\mid \, -y_i < \lambda} 1$.
Since $\lambda_0\geq \underline{y}$, immediately $\varphi'_+(\lambda_0)>0$. As
$y_i+\underline y\leq 0$ for all $i$, we have $\varphi(\underline y) = 0 < r$.
So, if $\varphi(\lambda_0)>r$, then $\lambda_0>\underline y$, in which case
$\varphi'_-(\lambda_0)>0$. Item 1 follows.

At iteration $k=0$, the Newtonian step $\tilde\lambda_k$ is accepted as the new
iterate $\lambda_1$ since $\tilde\lambda_0 = \lambda_0 -
{(\varphi(\lambda_0)-r)}/{\varphi'_+(\lambda_0)}\in (\underline \lambda,
\overline \lambda) = (\lambda_0,\infty)$ if $\varphi(\lambda_0)<r$ and
$\tilde\lambda_0 = \lambda_0 - {(\varphi(\lambda_0)-r)}/{\varphi'_-(\lambda_0)}
\in (\underline \lambda, \overline \lambda) = (-\infty,\lambda_0)$ if
$\varphi(\lambda_0)>r$. The fact that $\varphi$ is convex and non-decreasing
guarantees that $\varphi(\lambda_1) > r = \varphi(\lambda_*)$.
This implies $\lambda_1\geq \lambda_*$, and so $\varphi'_-(\lambda_1)>0$.
In the second iteration, $\underline\lambda$ remains unchanged, $\overline
\lambda\leftarrow\lambda_1$ and $\lambda_2 = \tilde\lambda_1 = \lambda_1 -
{(\varphi(\lambda_1)-r)}/{\varphi'_-(\lambda_1)}\in [\lambda_*,\lambda_1)
\subseteq (\underline\lambda,\overline\lambda)$. This repeats at every
iteration, leading to items 2, 3, and 4.
\end{proof}

\begin{algorithm}[ht]
	\caption{Semismooth Newton method for solving
		\eqref{PDelta}}\label{alg:purenewtonsimplex}
	\begin{algorithmic}[1]
		\Require the initial dual estimate $\lambda_0$
		\State Check, if necessary, whether
    $\lambda_0\geq \underline y = \min_i\{-y_i\}$ or not
    and update
    $\lambda_0 \leftarrow \max\{\lambda_0,\underline y\}$
    \label{alg:purenewtonsimplexinit}
		\If {$\varphi(\lambda_0)=r$}
		\State Stop and return $x(\lambda_0)$ as the solution of \eqref{PDelta}
		\EndIf
		\If {$\varphi(\lambda_0)<r$}
		\State $\lambda_1 = \lambda_0 - \frac{\varphi(\lambda_0)-r}{\varphi'_+(\lambda_0)}$
		\Else
		\State $\lambda_1 = \lambda_0 - \frac{\varphi(\lambda_0)-r}{\varphi'_-(\lambda_0)}$\label{alg:purenewtonsimplexstep}
		\EndIf
		\State Set $k\leftarrow 1$
		\If {$\varphi(\lambda_k)\leq r$}\label{alg:purenewtonsimplexstop}
		\State Stop and return $x(\lambda_k)$ as the solution of \eqref{PDelta}
		\EndIf
		\State Define $\lambda_{k+1} = \lambda_k - \frac{\varphi(\lambda_k)-r}{\varphi'_-(\lambda_k)}$, set $k\leftarrow k+1$ and go to line \ref{alg:purenewtonsimplexstop}
	\end{algorithmic}
\end{algorithm}

\subsection{Warm start and variable fixing for Algorithm 4}
\label{sec:warmstart}

Let us now consider the computation of the initial multiplier $\lambda_0$ for
Algorithm~\ref{alg:purenewtonsimplex}. If no prior information is available,
one possibility is to use Algorithm~\ref{alg:lambda0simplex} with $I = \{1,
\ldots, n \}$. Since this method refines the multiplier estimate using a
Gauss-Seidel approach, it generally provides a better approximation than a
direct application of formula~\eqref{CQKlambda0}. In particular, we know that
Algorithm~\ref{alg:lambda0simplex} is guaranteed to return $\lambda_0 \geq
\lambda_* > \underline y$, which avoids the need for the verification
performed in line~\ref{alg:purenewtonsimplexinit} of the specialized Newton
method. Moreover, this initialization identifies many variables that are fixed
at $0$ in the optimal solution, thereby reducing the problem size for the
subsequent Newton iterations if variable fixing is employed.

On the other hand, there are scenarios where approximate primal solutions are
readily available. This occurs, for example, when projections are performed
within the inner iterations of an outer algorithm or when a sequence of
closely related problems is solved. In these cases, the proximity of
successive points allows the previous projection to serve as an estimate for
the current one. This strategy was employed in~\citet{Cominetti2014} to
accelerate the solution of SVM problems.

In the~\eqref{PDelta} case, a warm start can be performed by using
Algorithm~\ref{alg:lambda0simplex} to estimate the initial multiplier. Let
$\bar x \geq 0$ be a good approximation of the optimal solution $x^* =
x(\lambda_*)$ of~\eqref{PDelta}. We can then define the set of candidates for
free variables as the support of the approximation, $\bar I_+ = \{i \mid \bar
x_i > 0 \}$, and use it as the initial set $I$ in
Algorithm~\ref{alg:lambda0simplex}. This approach assumes that the free
variables in $\bar x$ are likely to remain free in $x^*$, thereby yielding a
more accurate $\lambda_0$. This initialization maintains the property that
$\lambda_0 \geq \lambda_* \geq \underline y$. Furthermore, we can fix $x^*_i =
0$ whenever $y_i +\lambda \leq 0$, where $\lambda$ is the intermediate
approximation computed during the execution of
Algorithm~\ref{alg:lambda0simplex} that leads to the final value $\lambda_0$.

Finally, it is worth noting that we do not need to precompute $\bar I_+$
explicitly to apply Algorithm~\ref{alg:lambda0simplex} with $I = \bar I_+$.
This would require an additional pass over the data. Instead, we identify
these indices by testing whether the corresponding component of $\bar x$ is
positive before updating $\lambda$ (lines~\ref{alg:lambda0simplexlambda1},
\ref{alg:lambda0simplexlambda3}, and~\ref{alg:lambda0simplexlambda2} in
Algorithm~\ref{alg:lambda0simplex}).
In our
implementation, we maintain an auxiliary set of indices $J_+$ to store the
elements of $I_+$ used for computing the multiplier. In the rare event that
$J_+ = \emptyset$, the estimate~\eqref{lambdaJ} is undefined. In this case, we
simply set $\lambda = \max \{ r / n, -y_1 \} \geq \underline y$. We present
numerical results illustrating the computational benefits of this warm-start
strategy for Algorithm~\ref{alg:purenewtonsimplex} in
Section~\ref{sec:basispursuit}.

\section{Parallelizing CQK methods}
\label{sec:parallelNewton}

As stated in the introduction, Dai and Chen recently proposed a
parallelization strategy for simplex projection methods. They achieve this by
exploiting the fact that for a subproblem containing only a subset of the
variables, if a variable is identified as fixed at $0$, it necessarily remains
fixed at $0$ in the full problem instance~\cite[Proposition 4]{Dai2024}.
Parallelization is achieved by solving multiple subproblems, which partition
the original instance, to a coarse precision to identify the majority of the
fixed variables. Subsequently, a reduced version of the full instance is
solved. Since the vast majority of variables are expected to be fixed in a
typical solution, this final problem is usually very small and can be solved
in a fraction of the time required by the initial parallel phase. This
strategy can be applied to any method that performs variable fixing. In
particular, it results in a very effective variant of Condat's method, which
was shown to be the fastest in their implementation. This efficiency is
possible because the parallel phase, which identifies the fixed variables,
occurs before the Gauss-Seidel-style loop in Algorithm~\ref{alg:condat}, which
is inherently difficult to parallelize.

On the other hand, the pure Newton methods, Algorithms~\ref{alg:purenewtonCQK}
and~\ref{alg:purenewtonsimplex}, can be trivially parallelized in a
Jacobi-like fashion, as their primary steps contain no sequential
dependencies. The bulk of the computational effort is associated with
evaluating $\varphi$ and its derivatives. Such computations rely only on the
current multiplier $\lambda_k$, which is a single scalar that can be
efficiently shared among parallel workers. In fact, these methods fit
naturally into a MapReduce framework~\cite{Dean2008}. The \textit{map} phase
involves computing the individual summands that define the value of $\varphi$
and its derivatives across subsets of the data. The \textit{reduce} phase then
performs the summation to aggregate these values. Subsequently, a simple
multiplier update is performed to initiate the next iteration. This makes such
methods ideal for implementation on GPUs, where the memory indirection and
control-flow complexities often associated with the variable-fixing logic can
be difficult to implement efficiently~\cite{Goncalves2025}.

CPU implementations of the method can still benefit from variable fixing. In
this context, each worker in the MapReduce framework is responsible for a
subset of the data (a \emph{chunk}). At each iteration, the worker computes
the partial summands of $\varphi$ and its derivatives for the variables in its
chunk and performs local variable fixing. This remains possible because all
workers share a consistent multiplier estimate, which coordinates their local
computations. In our implementation, we also incorporate a mechanism to merge
chunks that become too small; this avoids the overhead of spawning or managing
threads for tasks that involve too little computational work.

The initial multiplier $\lambda_0$ can also be computed in parallel. On GPUs
or in the general~\eqref{CQK} case, it is preferable to use the formula
provided in~\eqref{CQKlambda0}, which fits naturally into a pure MapReduce
framework without requiring memory indirection. Conversely, for the simplex
case on CPUs, we can employ the more sophisticated
Algorithm~\ref{alg:lambda0simplex}. In this case, each worker executes the
algorithm in parallel on a distinct subset $I$ of the variables. Upon
completion, each thread will have identified its own subset $J_I$ of non-fixed
variables. The global $\lambda_0$ is then obtained via~\eqref{lambdaJ} using
the union $J = \cup_I J_I$. This aggregation is performed using the partial
sums already computed within each chunk, thereby avoiding redundant operations
when calculating the final $\lambda_0$.

Finally, the primal solution $x(\lambda_*)$ can also be computed in parallel,
as it is fundamentally a \textit{map} operation in the GPU implementation. In
the CPU implementation, which utilizes variable fixing, this calculation is
performed as soon as a variable is fixed. This is necessary because fixing a
variable to its bound can modify the right-hand side of the equality
constraint, $r$. However, in the specific case of the simplex, $x(\lambda_*)$
is calculated only at the end of the process; since fixed variables in the
simplex are always zero, they do not affect the right-hand side of the
constraint.

\section{Numerical tests}
\label{sec:numerical}

We have implemented the Newton algorithms in Julia in the {\tt NewtonCQK.jl}
package that can be installed using the language's standard package system.
The code is available at \url{https://github.com/pjssilva/NewtonCQK.jl}. It
has both CPU (with variable fixing) and GPU (pure MapReduce) variants to solve
the general~\eqref{CQK} problem and its special cases~\eqref{PDelta} and the
associated projection onto the $\ell_1$-ball. We compare these implementations
with the state-of-the-art methods and implementations for such problems,
namely, the original C implementation of the Newton method
of~\citet{Cominetti2014}, Condat's C implementation of his
method~\cite{Condat2016}, and the parallel implementation of~\citet{Dai2024}.
\citet{Cominetti2014} showed that their method outperforms highly specialized
methods based on median search, variable fixing, and secant steps
for~\eqref{CQK}, while \citet{Dai2024} reported Condat's method and their
parallel variation as the fastest serial and parallel methods, respectively,
among those they tested for projections onto simplices and $\ell_1$ balls.
Note that, in preliminary tests, general-purpose solvers, such as
CPLEX~\cite{CPLEX}, Gurobi~\cite{Gurobi}, and Hexaly~\cite{Hexaly}, and
generic projection codes, such as PPROJ~\cite{Hager2016}, did not prove
competitive. Moreover, their parallelization is not as efficient as the one
achieved by the highly specialized solvers, since general problems are not
easily decoupled.

The implementations of Algorithm~\ref{alg:purenewtonCQK} stop declaring
success if one of the following criteria is satisfied:
\begin{enumerate}
\item \label{stop:relsolve} $|\varphi(\lambda_k) - r| < \epsilon^{3/4} \Big(
\sum_{i=1}^n |b_i x(\lambda_k)| + |r|\Big)$ \hfill (relative to line
\ref{alg:purenewtonCQKstop} of Algorithm \ref{alg:purenewtonCQK});

\item \label{stop:nullstep} $| {(\varphi(\lambda_k)-r)}/{\varphi'(\lambda_k)}
| < \epsilon^{3/4}$ \ or \ $\lambda_{k-1} = \lambda_k$ \hfill (numerically
zero step);

\item \label{stop:bracket} $\overline\lambda - \underline\lambda <
\epsilon^{3/4} \max\{|\overline\lambda|, |\underline\lambda|\}$ \hfill (no
improvement is expected),
\end{enumerate}
where $\epsilon$ is the machine epsilon ($\epsilon\approx 2.22\times 10^{-16}$
for 64-bit floating-point and $\epsilon\approx 1.19\times 10^{-7}$ for
32-bit). Criteria~\ref{stop:relsolve} and~\ref{stop:bracket} appear in the
implementation from \citet{Cominetti2014}, except for the use of $10^{-12}$
in item~\ref{stop:relsolve} and $2.22\times 10^{-16}$ in item~\ref{stop:bracket}
instead of $\epsilon^{3/4}$. For Algorithm \ref{alg:purenewtonsimplex} and
related algorithms, only criteria~\ref{stop:nullstep} and~\ref{stop:bracket} are used.

We test our implementation against the following software:
\begin{itemize}
\item the sequential Newton method described in \citet{Cominetti2014} for
solving \eqref{CQK}. Its C code is freely available at
\url{www.ime.unicamp.br/~pjssilva/code}. We modified the code to use the
stopping criteria described above. We refer to this implementation as the
``CMS'' for short;

\item the sequential Condat's method (Algorithm \ref{alg:condat}) for
projecting onto a simplex, whose C code is freely available at
\url{lcondat.github.io/software.html};

\item the parallel Condat's method developed by \citet{Dai2024} for projecting
onto a simplex and an $\ell_1$ ball, whose Julia code is freely available at
\url{github.com/foreverdyz/Parallel-Simplex-Projection}.
\end{itemize}

All methods use a similar amount of memory. The additional memory required for
the parallel variations is at most proportional to the number of cores, which
is negligible when compared with problem sizes. Hence we focus the comparison
on running time. The CPU implementations were tested on a computer equipped
with two 24-core AMD EPYC 9254 processors, totaling 48 cores, and 1.5 TB RAM
under Debian 6.1. Experiments with GPU were performed on another
computer using an NVIDIA Tesla H100 NVL 94 GB. Table~\ref{tab:hardware}
presents the detailed hardware specification.

Let us define the nomenclature used in the following sections. When comparing
different methods or distinct implementations of the same algorithm, we
designate a base method and report the \textit{relative performance} of
alternative methods. We define \textit{relative performance} as the ratio of
the execution time of the base method to that of an alternative
(possibly parallel) implementation. Consequently, a value greater than 1
indicates that the alternative is faster, while a value less than 1 indicates
that the base method is more efficient. We utilize the term (parallel)
\textit{speedup} to assess how implementations scale with the number of
processing units. By using a single sequential method as the common numerator
for these ratios, we can compare the absolute performance of different
algorithms and their respective parallel behaviors within the same figure. An
increasing ratio indicates effective parallelization, while a decrease
suggests that computational overhead -- likely due to thread management or
communication bottlenecks -- outweighs the gains from concurrent execution.

\begin{table}
\begin{tabular*}{\columnwidth}{@{\extracolsep\fill}rll@{\extracolsep\fill}}
	\hline
	& CPU & GPU\\
	\hline
	Model & AMD EPYC 9254 24-Core & NVIDIA H100 NVL \\
	Cores & 48 divided into 2 equal processors & 16,896 CUDA cores \\
	Cache & L1: 96 KB, L2: 512 KB/core; L3: 32 MB & -- \\
	TFLOPS (FP64) & $\approx$ 1.3 -- 1.6 TFLOPS & $\approx$ 30 TFLOPS\\
	Memory & 1.5 TB DDR5 & 94 GB HBM3 \\
	Mem. bandwidth & $\approx 460.8$ GB/s & $\approx 3.9$ TB/s \\
	\hline
\end{tabular*}
\caption{Hardware specification.}\label{tab:hardware}
\end{table}

\subsection{Tests with general quadratic knapsack problems}

\subsubsection{Randomly generated instances}

Following \citet{Cominetti2014} and references therein, we considered three
classes of problems:
\begin{enumerate}
\item \textit{Uncorrelated}: \ $d_i, a_i, b_i \sim U[10,25]$ for each $i$;
\item \textit{Weakly correlated}: \ $b_i \sim U[10,25]$ and $d_i,a_i\sim
U[b_i-5,b_i+5]$ for each $i$;
\item \textit{Correlated}: \ $b_i \sim U[10,25]$ and $d_i = a_i = b_i+5$ for
each $i$,
\end{enumerate}
where $U[p,q]$ stands for the uniform distribution in the interval $[p,q]$.
For all classes, the $l_i$ and $u_i$ values were chosen as the minimum and
maximum between two numbers sorted i.i.d. from $U[10,25]$, and $r \sim U[b^tl,
b^tu]$. The random numbers were generated using the Julia package
\texttt{Distributions.jl}~\cite{Besancon2021}. In particular, we used its
default option, the algorithm~\texttt{Xoshiro256++}~\cite{Marsaglia2003}.

We generated 20 instances for dimensions ranging from 1,000 to
1,000,000,000. The runtime for each instance was measured using
\texttt{BenchmarkTools.jl}~\cite{Chen2016} to minimize measurement
errors. In particular, \texttt{BenchmarkTools.jl} was set to solve each
instance until 10,000 runs or 2 seconds had elapsed. The final value for each
instance is the minimum measured runtime, as it provides a robust estimator of
the actual value~\cite{Chen2016}. We then report for each dimension the median
runtime of its instances.

Table~\ref{tab:cqk_cms} presents the median CPU times for the Julia
implementation of Algorithm~\ref{alg:purenewtonCQK} with variable fixing,
available in~\texttt{NewtonCQK.jl}, alongside the original CMS implementation
written in C. Across various problem classes and sizes, the results
demonstrate that the Julia version is nearly as fast as its C counterpart.
This happens even though the Julia code can use a MapReduce framework
implemented in \texttt{Polyester.jl}~\cite{Polyester} that allows the same
codebase to execute in a multithreaded environment and exploit the multiple
cores of modern CPUs. This parallelization capability is not available in the
original serial CMS implementation. Moreover, the Julia code can work with any
floating-point type and performs basic sanity tests on the data (checking whether $d
> 0$, $b > 0$, and $l \leq u$) that are not present in the CMS implementation.

\begin{table}[!h]
\footnotesize
\begin{tabular*}{\columnwidth}{@{\extracolsep\fill}lllllll@{\extracolsep\fill}}
	\hline
	& \multicolumn{2}{l}{Uncorrelated} & \multicolumn{2}{l}{Weakly correlated} & \multicolumn{2}{l}{Correlated}\\
	$n$  & CMS time & rel perf & CMS time & rel perf & CMS time & rel perf\\
	\hline
	$10^{3}$ & 6.00e$-$03 &   0.9609 & 6.34e$-$03 &   0.9526 & 7.20e$-$03 &   0.9420 \\
	$10^{4}$ & 6.44e$-$02 &   0.9127 & 6.68e$-$02 &   0.9390 & 7.50e$-$02 &   0.9336 \\
	$10^{5}$ & 3.45e$+$00 &   1.0829 & 3.38e$+$00 &   1.0760 & 3.14e$+$00 &   1.0609 \\
	$10^{6}$ & 3.77e$+$01 &   1.0902 & 3.55e$+$01 &   1.0888 & 3.22e$+$01 &   1.0901 \\
	$10^{7}$ & 3.87e$+$02 &   1.0840 & 3.63e$+$02 &   1.0961 & 3.27e$+$02 &   1.0906 \\
	$10^{8}$ & 3.62e$+$03 &   1.0773 & 3.65e$+$03 &   1.0858 & 3.36e$+$03 &   1.0858 \\
	$10^{9}$ & 4.22e$+$04 &   1.0778 & 3.65e$+$04 &   1.0876 & 3.32e$+$04 &   1.0898 \\
	\hline
\end{tabular*}
\caption{Comparison between the sequential Algorithm~\ref{alg:purenewtonCQK}
	implemented in Julia and the CMS implementation for a single core. Times
	are in milliseconds and were taken as the median over the 20 generated
	instances. The columns labeled ``rel perf'' present the relative
	performance with \texttt{NewtonCQK.jl} as the base method and CMS as the
	alternative, that is, the ratio between the runtime of our implementation
	and that of CMS. Values above 1 indicate that CMS is faster.}
\label{tab:cqk_cms}
\end{table}

Next, we focus on the parallel scalability of the \texttt{NewtonCQK.jl}
implementation of Algorithm~\ref{alg:purenewtonCQK} with variable fixing on
CPUs. Figure~\ref{fig:cqkspeedup} illustrates the speedup relative to the
sequential version as the problem size increases. As expected, employing
multiple threads is not effective for small problems ($n = 1{,}000$). In
these cases, thread management overhead constitutes a significant portion of
the total execution time. Generally, as the number of variables increases, the
benefits of using additional threads become more pronounced, with the speedup
approaching linearity for the largest instances. In our benchmarks, parallel
execution becomes clearly beneficial for $n \geq 100{,}000$. Ultimately, users should
select the threading configuration that best aligns with their specific
problem scale and hardware architecture.

\begin{figure}[!ht]
	\centering
	\def\scale{0.48}
	\foreach \n in {1000,10000,100000,1000000,10000000,100000000,1000000000}{
	\subfigure{
		\includegraphics[width=\scale\textwidth]{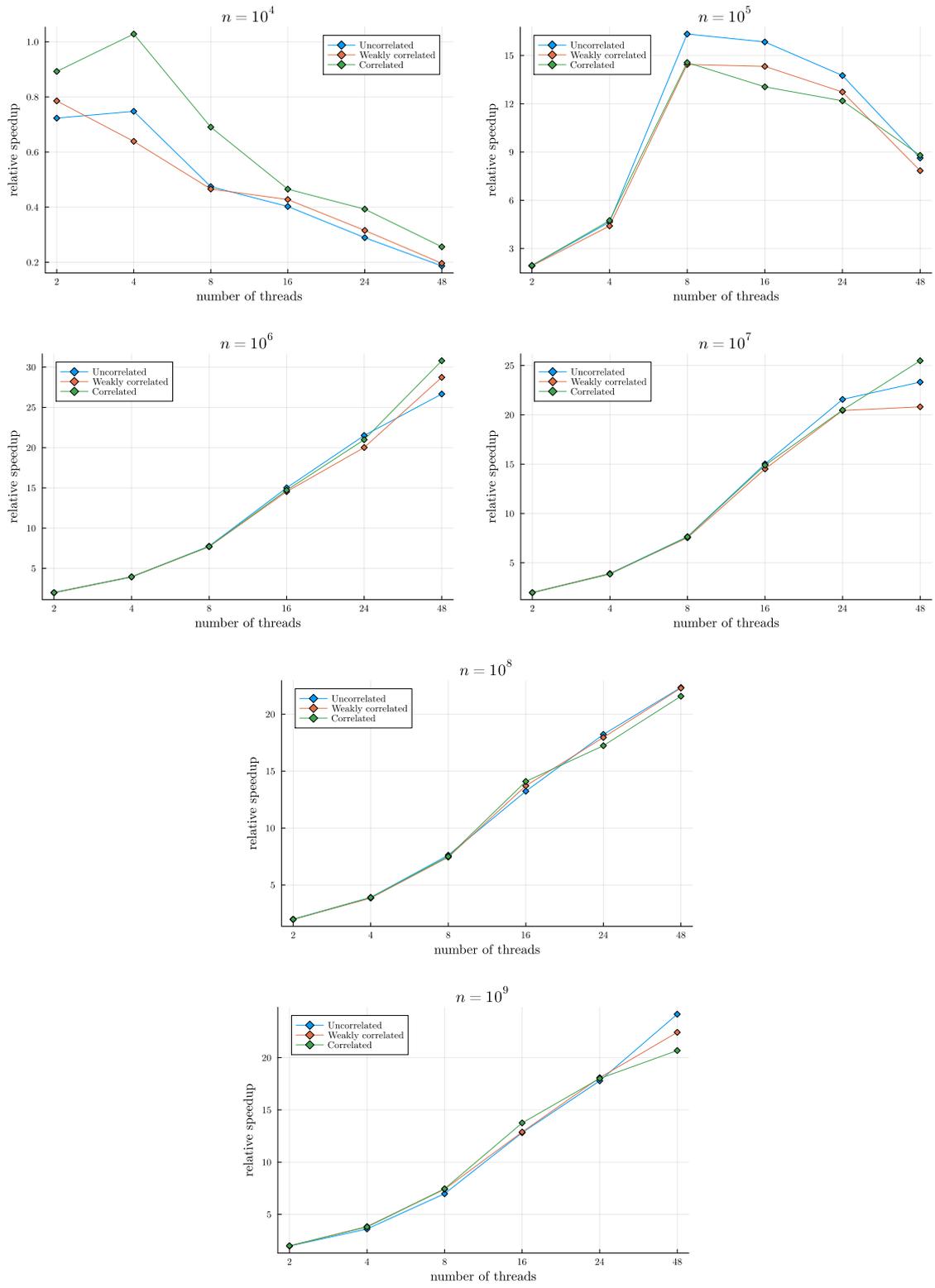}}
	}
	\caption{Speedups of parallel Algorithm \ref{alg:purenewtonCQK} with
	variable fixing relative to the sequential version on random instances.}
	\label{fig:cqkspeedup}
\end{figure}

Finally, we compare the CPU and GPU variants. It is important to note that the
GPU version does not utilize variable fixing. Table~\ref{tab:cqkuncorr}
presents these results, where the GPU columns indicate the relative
performance of the GPU version compared to the CPU implementation using
various thread counts. The values in parentheses represent the mean number of
iterations performed across the 20 generated instances. For large-scale
problems ($n \geq 10^7$), the GPU version is significantly faster,
outperforming even the CPU implementation running on all 48 threads. Notably,
for these problems, the CPU version using a single
thread and 48 threads is at least 229 and 7 times slower than the GPU
implementation, respectively.

\begin{table}[!ht]
\footnotesize
\begin{tabular*}{\columnwidth}{@{\extracolsep\fill}llllllll@{\extracolsep\fill}}
	\hline
	&  & \multicolumn{2}{l}{Uncorrelated} & \multicolumn{2}{l}{Weakly correlated} & \multicolumn{2}{l}{Correlated}\\
	$n$ & th  & CPU time & GPU rel perf & CPU time & GPU rel perf & CPU time & GPU rel perf\\
	\hline
	$10^{4}$ & 1 & 6e$-$02 (6.6) &   0.2 (5.8) & 6e$-$02 (5.6) &   0.2 (5.8) & 7e$-$02 (5.3) &   0.2 (5.6) \\
	$10^{4}$ & 2 & 5e$-$02 (6.6) &   0.2 (5.8) & 4e$-$02 (5.6) &   0.1 (5.8) & 5e$-$02 (5.3) &   0.1 (5.6) \\
	$10^{4}$ & 4 & 4e$-$02 (6.6) &   0.1 (5.8) & 4e$-$02 (5.6) &   0.1 (5.8) & 4e$-$02 (5.3) &   0.1 (5.6) \\
	$10^{4}$ & 8 & 4e$-$02 (6.6) &   0.2 (5.8) & 4e$-$02 (5.6) &   0.1 (5.8) & 4e$-$02 (5.3) &   0.1 (5.6) \\
	$10^{4}$ & 16 & 6e$-$02 (6.6) &   0.2 (5.8) & 5e$-$02 (5.6) &   0.2 (5.8) & 4e$-$02 (5.3) &   0.1 (5.6) \\
	$10^{4}$ & 24 & 7e$-$02 (6.6) &   0.2 (5.8) & 6e$-$02 (5.6) &   0.2 (5.8) & 5e$-$02 (5.3) &   0.2 (5.6) \\
	$10^{4}$ & 48 & 9e$-$02 (6.6) &   0.3 (5.8) & 1e$-$01 (5.6) &   0.3 (5.8) & 1e$-$01 (5.3) &   0.3 (5.6) \\
	\hline
	$10^{5}$ & 1 & 4e$+$00 (5.5) &  10.6 (5.8) & 4e$+$00 (6.5) &  10.1 (6.2) & 3e$+$00 (5.6) &   9.9 (5.7) \\
	$10^{5}$ & 2 & 2e$+$00 (5.5) &   5.2 (5.8) & 2e$+$00 (6.5) &   5.0 (6.2) & 2e$+$00 (5.6) &   4.8 (5.7) \\
	$10^{5}$ & 4 & 6e$-$01 (5.5) &   1.8 (5.8) & 8e$-$01 (6.5) &   2.1 (6.2) & 7e$-$01 (5.6) &   1.8 (5.7) \\
	$10^{5}$ & 8 & 1e$-$01 (5.5) &   0.3 (5.8) & 1e$-$01 (6.5) &   0.4 (6.2) & 1e$-$01 (5.6) &   0.4 (5.7) \\
	$10^{5}$ & 16 & 9e$-$02 (5.5) &   0.3 (5.8) & 1e$-$01 (6.5) &   0.3 (6.2) & 9e$-$02 (5.6) &   0.3 (5.7) \\
	$10^{5}$ & 24 & 9e$-$02 (5.5) &   0.2 (5.8) & 1e$-$01 (6.5) &   0.3 (6.2) & 9e$-$02 (5.6) &   0.3 (5.7) \\
	$10^{5}$ & 48 & 1e$-$01 (5.5) &   0.3 (5.8) & 1e$-$01 (6.5) &   0.4 (6.2) & 1e$-$01 (5.6) &   0.4 (5.7) \\
	\hline
	$10^{6}$ & 1 & 4e$+$01 (6.1) &  82.7 (6.2) & 4e$+$01 (6.4) &  82.1 (6.1) & 4e$+$01 (5.5) &  80.7 (5.7) \\
	$10^{6}$ & 2 & 2e$+$01 (6.1) &  41.1 (6.2) & 2e$+$01 (6.4) &  40.5 (6.1) & 2e$+$01 (5.5) &  39.8 (5.7) \\
	$10^{6}$ & 4 & 1e$+$01 (6.1) &  20.6 (6.2) & 1e$+$01 (6.4) &  20.5 (6.1) & 9e$+$00 (5.5) &  20.4 (5.7) \\
	$10^{6}$ & 8 & 5e$+$00 (6.1) &  10.3 (6.2) & 5e$+$00 (6.4) &  10.2 (6.1) & 4e$+$00 (5.5) &  10.2 (5.7) \\
	$10^{6}$ & 16 & 2e$+$00 (6.1) &   5.2 (6.2) & 2e$+$00 (6.4) &   5.1 (6.1) & 2e$+$00 (5.5) &   5.1 (5.7) \\
	$10^{6}$ & 24 & 1e$+$00 (6.1) &   3.2 (6.2) & 2e$+$00 (6.4) &   3.3 (6.1) & 1e$+$00 (5.5) &   3.2 (5.7) \\
	$10^{6}$ & 48 & 4e$-$01 (6.1) &   0.7 (6.2) & 5e$-$01 (6.4) &   1.0 (6.1) & 3e$-$01 (5.5) &   0.7 (5.7) \\
	\hline
	$10^{7}$ & 1 & 4e$+$02 (5.8) & 260.8 (6.2) & 4e$+$02 (6.2) & 254.3 (6.1) & 4e$+$02 (5.7) & 229.8 (6.1) \\
	$10^{7}$ & 2 & 2e$+$02 (5.8) & 139.1 (6.2) & 2e$+$02 (6.2) & 127.7 (6.1) & 2e$+$02 (5.7) & 115.6 (6.1) \\
	$10^{7}$ & 4 & 1e$+$02 (5.8) &  68.5 (6.2) & 1e$+$02 (6.2) &  66.9 (6.1) & 9e$+$01 (5.7) &  60.4 (6.1) \\
	$10^{7}$ & 8 & 6e$+$01 (5.8) &  36.6 (6.2) & 5e$+$01 (6.2) &  33.3 (6.1) & 5e$+$01 (5.7) &  30.2 (6.1) \\
	$10^{7}$ & 16 & 3e$+$01 (5.8) &  18.7 (6.2) & 3e$+$01 (6.2) &  17.4 (6.1) & 2e$+$01 (5.7) &  15.7 (6.1) \\
	$10^{7}$ & 24 & 2e$+$01 (5.8) &  12.8 (6.2) & 2e$+$01 (6.2) &  12.0 (6.1) & 2e$+$01 (5.7) &  11.1 (6.1) \\
	$10^{7}$ & 48 & 1e$+$01 (5.8) &   7.8 (6.2) & 1e$+$01 (6.2) &   7.9 (6.1) & 1e$+$01 (5.7) &   7.0 (6.1) \\
	\hline
	$10^{8}$ & 1 & 4e$+$03 (6.0) & 298.1 (6.5) & 4e$+$03 (6.0) & 322.6 (6.0) & 4e$+$03 (6.0) & 340.8 (5.4) \\
	$10^{8}$ & 2 & 2e$+$03 (6.0) & 152.0 (6.5) & 2e$+$03 (6.0) & 165.5 (6.0) & 2e$+$03 (6.0) & 172.1 (5.4) \\
	$10^{8}$ & 4 & 1e$+$03 (6.0) &  78.1 (6.5) & 1e$+$03 (6.0) &  85.3 (6.0) & 1e$+$03 (6.0) &  90.6 (5.4) \\
	$10^{8}$ & 8 & 6e$+$02 (6.0) &  41.8 (6.5) & 5e$+$02 (6.0) &  43.7 (6.0) & 5e$+$02 (6.0) &  46.5 (5.4) \\
	$10^{8}$ & 16 & 3e$+$02 (6.0) &  22.3 (6.5) & 3e$+$02 (6.0) &  24.3 (6.0) & 3e$+$02 (6.0) &  26.0 (5.4) \\
	$10^{8}$ & 24 & 2e$+$02 (6.0) &  17.4 (6.5) & 2e$+$02 (6.0) &  17.8 (6.0) & 2e$+$02 (6.0) &  19.5 (5.4) \\
	$10^{8}$ & 48 & 2e$+$02 (6.0) &  12.8 (6.5) & 2e$+$02 (6.0) &  13.9 (6.0) & 2e$+$02 (6.0) &  15.9 (5.4) \\
	\hline
	$10^{9}$ & 1 & 4e$+$04 (6.7) & 345.1 (6.2) & 4e$+$04 (5.9) & 325.2 (6.0) & 4e$+$04 (6.0) & 304.5 (5.8) \\
	$10^{9}$ & 2 & 2e$+$04 (6.7) & 174.3 (6.2) & 2e$+$04 (5.9) & 165.1 (6.0) & 2e$+$04 (6.0) & 156.7 (5.8) \\
	$10^{9}$ & 4 & 1e$+$04 (6.7) &  91.1 (6.2) & 1e$+$04 (5.9) &  86.5 (6.0) & 1e$+$04 (6.0) &  81.3 (5.8) \\
	$10^{9}$ & 8 & 6e$+$03 (6.7) &  46.1 (6.2) & 5e$+$03 (5.9) &  44.7 (6.0) & 5e$+$03 (6.0) &  41.8 (5.8) \\
	$10^{9}$ & 16 & 3e$+$03 (6.7) &  25.8 (6.2) & 3e$+$03 (5.9) &  25.2 (6.0) & 3e$+$03 (6.0) &  22.6 (5.8) \\
	$10^{9}$ & 24 & 3e$+$03 (6.7) &  19.2 (6.2) & 2e$+$03 (5.9) &  18.0 (6.0) & 2e$+$03 (6.0) &  17.5 (5.8) \\
	$10^{9}$ & 48 & 2e$+$03 (6.7) &  14.5 (6.2) & 2e$+$03 (5.9) &  13.8 (6.0) & 2e$+$03 (6.0) &  14.1 (5.8) \\
	\hline
\end{tabular*}
\caption{Comparison between CPU and GPU versions of Algorithm
\ref{alg:purenewtonCQK} implemented in \texttt{NewtonCQK.jl} using 64 bit
floating-point. Times are in milliseconds, and were taken as the median over
the 20 generated instances. The values between parentheses are the mean of
iterations.}
\label{tab:cqkuncorr}
\end{table}

\subsubsection{Support vector machines}
\label{sec:svm}

Following \citet{Cominetti2014}, we also evaluated our methods on
non-synthetic datasets by solving soft-margin SVM binary classification
problems. In this context, CQK subproblems arise naturally when applying
projected gradient methods to solve the dual formulation of the SVM. A brief
description of this setting follows.

Given a set of labeled data
\[
    T = \{(z^i, y_i) \mid z^i \in \R^m, \ y_i \in \{-1,1\}, \ i=1,\ldots,n\}
\]
we seek a classification function $F: \R^m \to \{-1,1\}$ of the form
\[
    F(z) = \text{sgn}\left( \sum_{i=1}^n x_i^* y_i K(z,z^i) + b_* \right),
\]
where $x^* \in \R^n$ and $b_* \in \R$ are the optimal
parameters to be determined, and $K: \R^m \times \R^m \to
\R$ is a fixed kernel function. In our experiments, we employ the
standard Gaussian kernel $K(p,q) = \exp(-\gamma \|p-q\|^2)$, where $\gamma >
0$ is a specified hyperparameter.

The vector $x^*$ is a solution to the dual problem~\cite{Burges1998}
\begin{equation*}
    \min_x \ f(x) = \frac{1}{2} x^t Hx - e^tx \quad
    \text{s.t.} \quad
    y^tx = 0, \quad 0 \leq x \leq Ce,
\end{equation*}
where $e$ is a vector of ones, $C > 0$ is a hyperparameter, and $H$ is the $n
\times n$ symmetric matrix with entries $H_{ij} = y_i y_j K(z^i, z^j)$. Since
$H$ is not diagonal, the problem above is not in the form of~\eqref{CQK}.
However, a projected gradient-type iteration
\begin{equation}
    x^{k+1} = x^k + t_k d^k, \qquad
    d^k = P_{\Omega}(x^k - \lambda_k \nabla f(x^k)) - x^k,
    \label{svm_graditer}
\end{equation}
where $t_k, \lambda_k > 0$ and $\Omega = \{x \in \mathbb{R}^n \mid y^tx = 0, \
0 \leq x \leq Ce\}$, requires computing the projection of $x^k - \lambda_k
\nabla f(x^k)$ onto $\Omega$. This corresponds exactly to solving the CQK
subproblem
\begin{equation}\label{svmsubprob}
    \min_x \ \frac{1}{2} x^t Ix -(x^k - \lambda_k \nabla f(x^k))^tx \quad
    \text{s.t.} \quad
    y^tx = 0, \quad 0 \leq x \leq Ce.
\end{equation}
The gradient-type method employed in our experiments is the Spectral Projected
Gradient (SPG) method~\cite{Birgin2001,Birgin2000}. This method was selected
due to its excellent practical performance in solving large-scale constrained
problems. The stopping criterion is defined by the infinity norm of the
projected gradient at a given iterate being smaller than $10^{-4}$.

Following the approach in \citet{Cominetti2014}, this section primarily aims
to verify if utilizing the previous iterate $x^k$ to warm start
Algorithm~\ref{alg:purenewtonCQK} yields significant performance gains, as
discussed in Section~\ref{sec:warmstart}. We use subsets of the original
datasets because our implementation stores the dense matrix $H$, whose size
grows quadratically with $n$. We also note that, in this application, the cost
of each SPG iteration is dominated by the product of $H$ by a vector, which
requires $O(n^2)$ operations, while the projection requires only $O(n)$.
Nevertheless, in a parallel or GPU implementation of SPG, where this product
is also computed in parallel, it is natural to use a projection that runs in
the same environment. On GPUs, this also avoids transferring the iterates
between the GPU and the CPU at each iteration. Hence, we also investigate the
behavior and efficiency of the parallel implementations in this practical
application context.

For these experiments, we utilized two datasets:
\begin{enumerate}
    \item The \texttt{MNIST} dataset of handwritten digits, as previously
    used in \citet{Cominetti2014}. We selected all 5,851 training samples
    labeled as digit 8 and the first 5,851 occurrences of other digits,
    totaling 11,702 samples. The objective is to design a classifier $F$ to
    determine whether a $20 \times 20$ pixel image represents the digit 8;

    \item The Diabetes Health Indicators dataset from the U.S. Centers for
    Disease Control and Prevention (``\texttt{cdc\_diabetes}''), which
    contains health and lifestyle information for 253,680 individuals. Following
    a similar sampling strategy, we selected the first 10,000 occurrences of
    ``prediabetes or diabetes'' and the first 10,000 occurrences of
    ``no diabetes'', resulting in a total of 20,000 samples.
\end{enumerate}
The datasets were obtained using the Julia package \texttt{OpenML.jl}
(\url{https://github.com/JuliaAI/OpenML.jl}) under IDs 554 and 46,598,
respectively. Detailed descriptions are available at \url{https://www.openml.org}.

Regarding the hyperparameters, we set $\gamma = 0.007$ and $C = 5.0$ for the
\texttt{MNIST} dataset, and $\gamma = 0.5$ and $C = 10.0$ for
\texttt{cdc\_diabetes}. These values were obtained using cross-validation and
yield high-accuracy models on the training data, with a misclassification rate
of less than 0.5\%.

Figure~\ref{fig:svmspeedup} illustrates the relative performance of the
\texttt{NewtonCQK.jl} and CMS implementations of
Algorithm~\ref{alg:purenewtonCQK} with variable fixing and warm start,
compared to the \texttt{NewtonCQK.jl} implementation without warm start
(initialized according to~\eqref{CQKlambda0}). The figure depicts the behavior
over the first and last 100 SPG iterations, omitting the initial three
iterations due to highly discrepant values. We observe that using $x^k$ as the
initial guess significantly improves performance, especially during the final
iterations. As discussed in Section~\ref{sec:warmstart}, this is expected
because the subproblems~\eqref{svmsubprob} undergo only minor adjustments;
thus, the set $\{i \mid 0 < x_i^k < C\}$ used to compute $\lambda_0$ remains
mostly constant, identifying the active set at the solution. Our results
corroborate the findings in \citet{Cominetti2014}. Finally,
\texttt{NewtonCQK.jl} and CMS exhibit similar behavior, as they implement
essentially the same method, although results for the smaller \texttt{MNIST}
instance show a slight advantage for the CMS implementation.

\begin{figure}[ht]
\centering
\def\scale{0.49}
\foreach \problem in {mnist\_784,cdc\_diabetes}{
\subfigure{\includegraphics[width=\scale\textwidth]{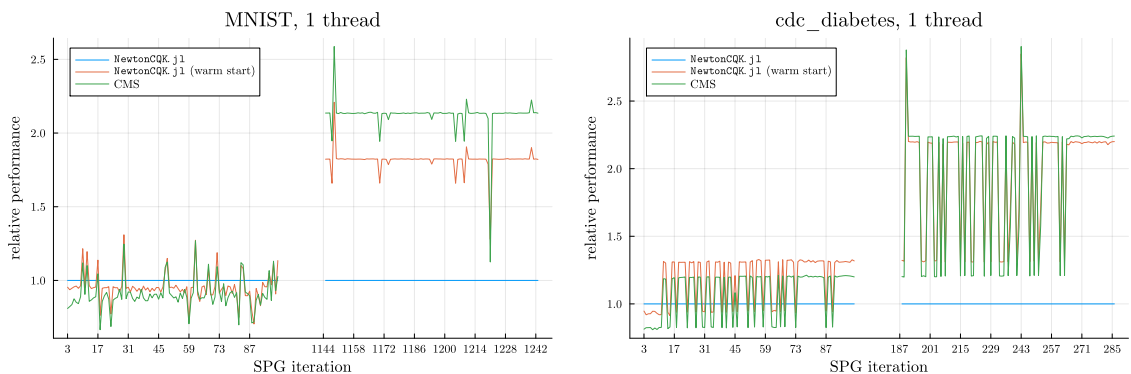}}
}
\caption{Relative performance of Algorithm~\ref{alg:purenewtonCQK} from
\texttt{NewtonCQK.jl} and CMS using variable fixing and warm starts to solve
\eqref{svmsubprob} along the SPG iterations. The baseline (horizontal blue
line at 1) is the performance of the~\texttt{NewtonCQK.jl} without warm start.}
\label{fig:svmspeedup}
\end{figure}

To conclude this subsection, we present the speedup of the parallel
implementation in \texttt{NewtonCQK.jl} in
Figure~\ref{fig:svmspeedupparallel}. For \texttt{MNIST}
multithreading proves beneficial only for a modest number of threads --
specifically up to 8 -- given their relatively small scale.
For the larger \texttt{cdc\_diabetes} instance, benefits are observed
when using up to 24 threads.
It is also important to
emphasize that the initial multiplier $\lambda_0$ computed
from~\eqref{CQKlambda0} is identical in both the multithreaded and sequential
implementations. This consistency contrasts with the projections onto a
simplex and an $\ell_1$-ball, which we address in the following subsections.

\begin{figure}[ht]
\centering
\def\scale{0.49}
\foreach \problem in {mnist\_784,cdc\_diabetes}{
	\subfigure{\includegraphics[width=\scale\textwidth]{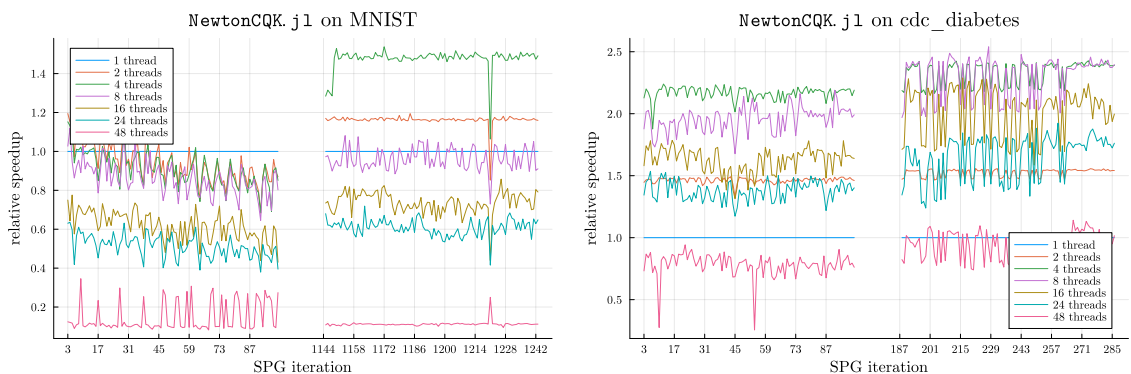}}
}
\caption{Speedups of parallel Algorithm \ref{alg:purenewtonCQK} for solving
\eqref{svmsubprob} in relation to the sequential algorithm. All algorithms
were executed using the current SPG iterate as warm start.}
\label{fig:svmspeedupparallel}
\end{figure}

\subsection{Projection onto simplex and $\ell_1$ ball}
\label{sec:simplex}

\subsubsection{Randomly generated instances}

In this section, we present numerical results for problem~\eqref{PDelta}.
Following the methodology in \citet{Dai2024} and references therein, we
consider three classes of test problems based on the distribution of
the components of $y$:
\begin{enumerate}
    \item $y_i \sim U(0,1)$ for each $i$;
    \item $y_i \sim N(0,1)$, the standard normal distribution;
    \item $y_i \sim N(0,10^{-3})$ for each $i$.
\end{enumerate}
As in the previous experiments, the problem dimension $n$ ranges from $10^3$
to $10^9$, with 20 instances generated for each configuration. In the rare
event that a value $y_i = 0$ is generated, the entire instance is
randomized again.

We compare below the implementation of Algorithm~\ref{alg:purenewtonsimplex}
from \texttt{NewtonCQK.jl} with Condat's original implementation of his method
and with the implementation by Dai and Chen, which also features a parallel
variant. It is important to point out that Condat's implementation returns a
dense vector as the solution, even when the majority of its entries are zero
(at the lower bound). In contrast, Dai and Chen exploit the expected sparsity
of the solution to avoid allocating and filling a large dense vector,
returning a sparse representation instead. The \texttt{NewtonCQK.jl}
implementation allows the user to specify the desired return type, hence we
present results for both options.

Figures~\ref{fig:simplexspeedup1} and~\ref{fig:simplexspeedup3} present the
speedup results for problem types 1 and 3; results for
type 2 are omitted as they are qualitatively similar to those of type 1.
Condat's method serves as the baseline for all comparisons. By including
results for the sequential implementations (running on a single thread)
within these figures, we facilitate a direct comparison of the algorithmic
efficiency of each method on a single core alongside their parallel
scaling behavior.

\begin{figure}[ht]
\centering
\def\scale{0.49}
\foreach \n in {1000,10000,100000,1000000,10000000,100000000,1000000000}{
	\subfigure{\includegraphics[width=\scale\textwidth]{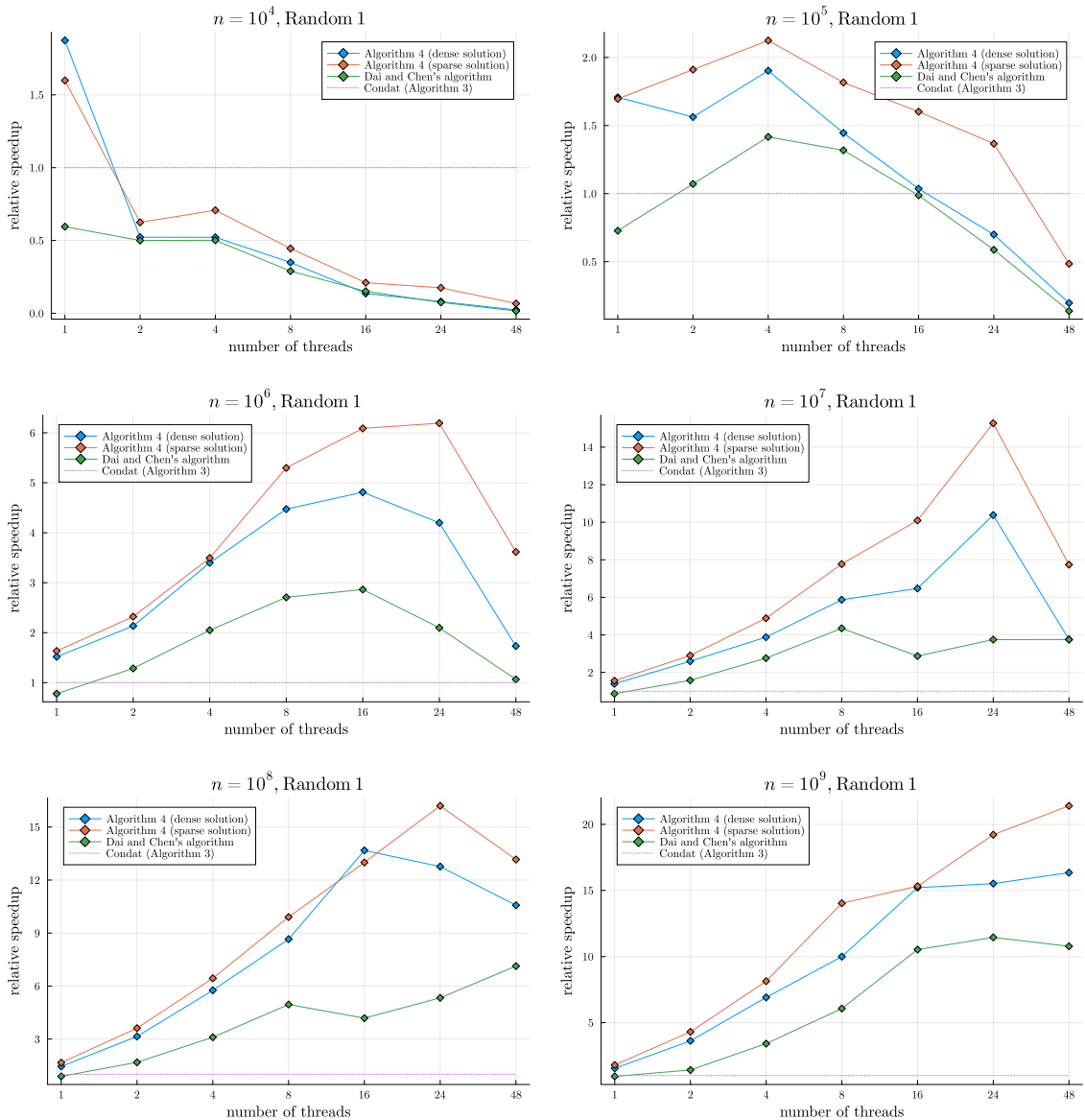}}
}
\caption{Speedups of Algorithm~\ref{alg:purenewtonsimplex} implemented
	in~\texttt{NewtonCQK.jl} (CPU versions returning dense and sparse
	solutions), Dai and Chen's (sparse solution) and Condat's (dense solution)
	algorithms on \eqref{PDelta}, instances of type~1. The base algorithm is
	of Condat, which is represented by the horizontal line.}
\label{fig:simplexspeedup1}
\end{figure}

\begin{figure}[ht]
\centering
\def\scale{0.49}
\foreach \n in {1000,10000,100000,1000000,10000000,100000000,1000000000}{
	\subfigure{\includegraphics[width=\scale\textwidth]{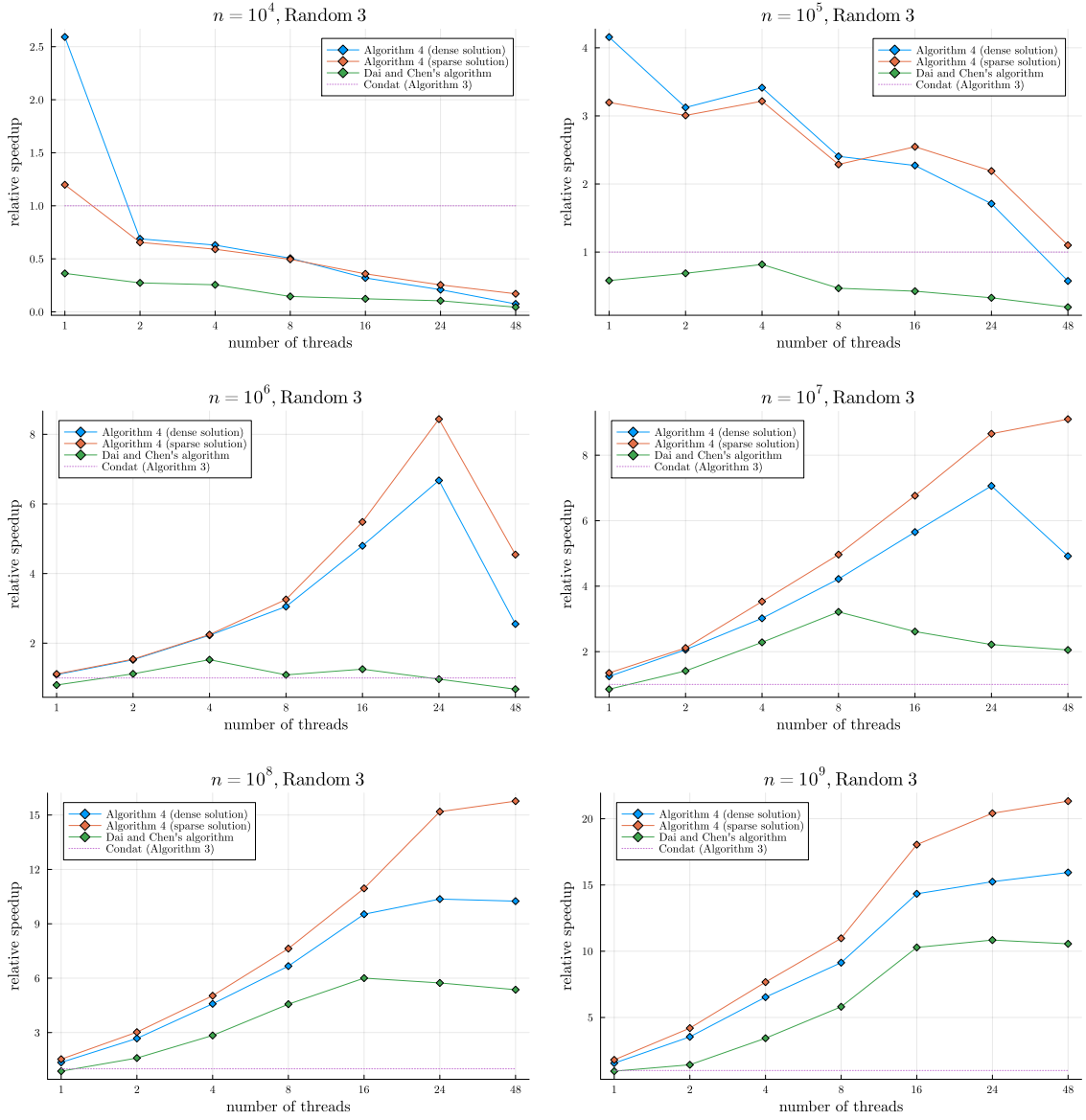}}
}
\caption{Speedups of Algorithm~\ref{alg:purenewtonsimplex} implemented
	in~\texttt{NewtonCQK.jl} (CPU versions returning dense and sparse
	solutions), Dai and Chen's (sparse solution) and Condat's (dense solution)
	algorithms on \eqref{PDelta}, instances of type~3. The base algorithm is
	of Condat, which is represented by the horizontal line.}
\label{fig:simplexspeedup3}
\end{figure}

We observe that both \texttt{NewtonCQK.jl} implementations of
Algorithm~\ref{alg:purenewtonsimplex} (dense and sparse output) outperform
Condat's original C implementation even in single-threaded mode, except for
the sparse output on instances of type~3 with $n = 10^3$, whose solutions have
about 80\% of nonzero entries. Regarding parallel performance, we again
observe no gain for small-scale problems ($n \leq 10^4$). For larger
instances, the speedups achieved by the \texttt{NewtonCQK.jl} variant
returning sparse vectors are consistently higher than those attained by Dai
and Chen's algorithm. Notably, even the \texttt{NewtonCQK.jl} variant that
produces dense vectors outperforms Dai and Chen's method on large-scale
problems. For $n = 10^9$, the performance improvement when increasing the
thread count from 16 to 24 is less pronounced. This suggests that for very
large problem sizes, the parallel workers may be reaching a point of
diminishing returns, where computational overhead and memory access patterns
become the primary limiting factors.

It is important to emphasize that the performance advantages demonstrated by
the \texttt{NewtonCQK.jl} implementation stem not only from algorithmic
differences -- specifically the pure Newton approach versus variable fixing --
but also from strategic choices in the data structures used to represent free
variables. Condat's and Dai and Chen's methods maintain the $y_i$ values of
the unfixed variables throughout the iterations. In contrast,
\texttt{NewtonCQK.jl}, following the logic of the CMS code for the general
problem~\eqref{CQK}, tracks the indices of the free variables instead. The
former approach is advantageous for maintaining data compactness in memory
during the iterative process; however, it incurs a computational penalty at
the conclusion, as a full scan of the problem data is required to construct
the final output. The latter approach may result in more dispersed memory
access during iterations, but it retains explicit knowledge of the few
potentially non-zero values at the solution, thereby avoiding a costly final
full scan. To quantify this effect, we modified Condat's C code to track the
indices of the free variables, keeping the rest of the method unchanged, so
that it can also return sparse solutions. On a single core and for $n \geq
10^7$, the modified code returning sparse solutions achieved speedups between
1.14 and 1.31 over the original one. This gain amounts to roughly 40\% to 50\%
of the time saved by the sequential \texttt{NewtonCQK.jl} variant returning
sparse solutions. When dense solutions are returned, the whole vector must
still be written, and the modified code was not faster than the original one.
Hence, in this case, the advantage of \texttt{NewtonCQK.jl} is not explained
by the data structures alone, but is influenced by other factors, such as the
way the multiplier is computed and differences between the implementations and
compilers.

Finally, Table~\ref{tab:simplexgpu} presents the results for the GPU MapReduce
implementation of Algorithm~\ref{alg:purenewtonsimplex}. We recall that due to
architectural constraints, this implementation utilizes the standard
multiplier initialization derived from~\eqref{CQKlambda0} and cannot
incorporate the algorithmic refinements described in
Section~\ref{sec:newtonsimplex}. Consequently, the GPU only consistently
outperforms the single-threaded CPU implementation once the problem size
reaches $n \geq 10^7$. For $n \geq 10^8$, the GPU version is 9 to 14 times
faster than the single-threaded CPU implementation and outperforms the
8-threaded CPU version, but its performance is comparable to that of the CPU
version with 16 threads. This behavior has two main causes. For $n \leq 10^7$,
each GPU iteration takes between 45 and 90~$\mu$s, almost independently of
$n$, which indicates that the running time is dominated by fixed costs, such
as kernel launches and synchronizations, and not by the computations over the
data. For $n \geq 10^8$, the GPU version is limited by the memory bandwidth:
since it does not fix variables, each iteration processes all $n$ variables,
while the CPU version iterates only over the variables that were not fixed
during the initialization. In both cases, the effect is amplified by the
iteration count, as the GPU version can require substantially more iterations
than the Newton method using the initialization in
Algorithm~\ref{alg:lambda0simplex}. This further demonstrates the
effectiveness of the alternative multiplier initialization introduced by
Condat. Nevertheless, a GPU implementation remains highly valuable for
integration into workflows that reside entirely on the GPU, as it eliminates
costly data transfers between the GPU and CPU, as seen
in~\citet{Behling2025,Goncalves2025}.

\begin{table}[!ht]
\footnotesize
\setlength{\tabcolsep}{5pt}
\begin{tabular*}{\columnwidth}{@{\extracolsep\fill}llllllll@{\extracolsep\fill}}
	\hline
	&  & \multicolumn{2}{l}{Random 1} & \multicolumn{2}{l}{Random 2} & \multicolumn{2}{l}{Random 3}\\
	$n$ & th  & CPU time & GPU rel perf & CPU time & GPU rel perf & CPU time & GPU rel perf\\
	\hline
	$10^{3}$ & 1 & 1e$-$03 (5.0) &   0.0 (8.2) & 6e$-$04 (2.8) &   0.0 (9.7) & 3e$-$03 (3.1) &   0.0 (4.4) \\
	$10^{3}$ & 2 & 4e$-$03 (5.3) &   0.0 (8.2) & 3e$-$03 (3.4) &   0.0 (9.7) & 9e$-$03 (3.4) &   0.0 (4.4) \\
	$10^{3}$ & 4 & 6e$-$03 (5.7) &   0.0 (8.2) & 6e$-$03 (4.3) &   0.0 (9.7) & 1e$-$02 (3.4) &   0.1 (4.4) \\
	$10^{3}$ & 8 & 1e$-$02 (6.0) &   0.0 (8.2) & 9e$-$03 (5.2) &   0.0 (9.7) & 2e$-$02 (3.4) &   0.1 (4.4) \\
	$10^{3}$ & 16 & 2e$-$02 (6.2) &   0.1 (8.2) & 1e$-$02 (5.5) &   0.0 (9.7) & 2e$-$02 (3.4) &   0.1 (4.4) \\
	$10^{3}$ & 24 & 3e$-$02 (6.3) &   0.1 (8.2) & 3e$-$02 (6.0) &   0.1 (9.7) & 3e$-$02 (3.4) &   0.2 (4.4) \\
	$10^{3}$ & 48 & 5e$-$02 (6.8) &   0.2 (8.2) & 4e$-$02 (6.5) &   0.1 (9.7) & 5e$-$02 (3.4) &   0.3 (4.4) \\
	\hline
	$10^{4}$ & 1 & 7e$-$03 (5.8) &   0.0 (10.2) & 6e$-$03 (3.1) &   0.0 (11.9) & 2e$-$02 (5.1) &   0.1 (7.0) \\
	$10^{4}$ & 2 & 9e$-$03 (6.1) &   0.0 (10.2) & 5e$-$03 (4.0) &   0.0 (11.9) & 3e$-$02 (5.3) &   0.1 (7.0) \\
	$10^{4}$ & 4 & 1e$-$02 (6.3) &   0.0 (10.2) & 8e$-$03 (4.9) &   0.0 (11.9) & 4e$-$02 (5.9) &   0.1 (7.0) \\
	$10^{4}$ & 8 & 1e$-$02 (6.9) &   0.0 (10.2) & 1e$-$02 (5.5) &   0.0 (11.9) & 4e$-$02 (6.0) &   0.1 (7.0) \\
	$10^{4}$ & 16 & 2e$-$02 (7.1) &   0.0 (10.2) & 2e$-$02 (6.1) &   0.0 (11.9) & 5e$-$02 (6.0) &   0.1 (7.0) \\
	$10^{4}$ & 24 & 3e$-$02 (7.2) &   0.1 (10.2) & 2e$-$02 (6.5) &   0.0 (11.9) & 5e$-$02 (6.0) &   0.1 (7.0) \\
	$10^{4}$ & 48 & 6e$-$02 (7.8) &   0.1 (10.2) & 4e$-$02 (7.1) &   0.1 (11.9) & 8e$-$02 (6.0) &   0.2 (7.0) \\
	\hline
	$10^{5}$ & 1 & 8e$-$02 (6.8) &   0.1 (12.2) & 6e$-$02 (3.6) &   0.1 (14.1) & 1e$-$01 (6.2) &   0.2 (9.1) \\
	$10^{5}$ & 2 & 5e$-$02 (7.0) &   0.1 (12.2) & 3e$-$02 (4.2) &   0.0 (14.1) & 8e$-$02 (6.8) &   0.2 (9.1) \\
	$10^{5}$ & 4 & 4e$-$02 (7.3) &   0.1 (12.2) & 2e$-$02 (5.2) &   0.0 (14.1) & 8e$-$02 (7.0) &   0.2 (9.1) \\
	$10^{5}$ & 8 & 4e$-$02 (7.8) &   0.1 (12.2) & 2e$-$02 (5.8) &   0.0 (14.1) & 7e$-$02 (7.1) &   0.1 (9.1) \\
	$10^{5}$ & 16 & 4e$-$02 (8.1) &   0.1 (12.2) & 2e$-$02 (6.6) &   0.0 (14.1) & 8e$-$02 (7.6) &   0.2 (9.1) \\
	$10^{5}$ & 24 & 4e$-$02 (8.2) &   0.1 (12.2) & 2e$-$02 (6.9) &   0.0 (14.1) & 9e$-$02 (7.9) &   0.2 (9.1) \\
	$10^{5}$ & 48 & 8e$-$02 (8.6) &   0.1 (12.2) & 4e$-$02 (7.4) &   0.0 (14.1) & 1e$-$01 (8.0) &   0.2 (9.1) \\
	\hline
	$10^{6}$ & 1 & 8e$-$01 (7.7) &   0.8 (14.0) & 6e$-$01 (4.3) &   0.5 (16.6) & 2e$+$00 (7.8) &   3.0 (11.6) \\
	$10^{6}$ & 2 & 5e$-$01 (8.0) &   0.5 (14.0) & 2e$-$01 (5.0) &   0.2 (16.6) & 1e$+$00 (8.0) &   2.1 (11.6) \\
	$10^{6}$ & 4 & 3e$-$01 (8.1) &   0.3 (14.0) & 2e$-$01 (5.7) &   0.1 (16.6) & 9e$-$01 (8.1) &   1.4 (11.6) \\
	$10^{6}$ & 8 & 2e$-$01 (8.5) &   0.2 (14.0) & 9e$-$02 (6.5) &   0.1 (16.6) & 4e$-$01 (8.8) &   0.7 (11.6) \\
	$10^{6}$ & 16 & 1e$-$01 (8.9) &   0.1 (14.0) & 5e$-$02 (7.0) &   0.0 (16.6) & 2e$-$01 (9.0) &   0.2 (11.6) \\
	$10^{6}$ & 24 & 1e$-$01 (9.0) &   0.1 (14.0) & 5e$-$02 (7.3) &   0.0 (16.6) & 2e$-$01 (9.0) &   0.2 (11.6) \\
	$10^{6}$ & 48 & 1e$-$01 (9.2) &   0.1 (14.0) & 5e$-$02 (8.1) &   0.0 (16.6) & 2e$-$01 (9.2) &   0.2 (11.6) \\
	\hline
	$10^{7}$ & 1 & 9e$+$00 (8.2) &   6.4 (16.0) & 7e$+$00 (4.4) &   4.3 (18.9) & 1e$+$01 (8.8) &  10.5 (14.0) \\
	$10^{7}$ & 2 & 5e$+$00 (8.6) &   3.5 (16.0) & 3e$+$00 (5.3) &   1.8 (18.9) & 8e$+$00 (9.0) &   6.7 (14.0) \\
	$10^{7}$ & 4 & 2e$+$00 (9.0) &   1.8 (16.0) & 2e$+$00 (6.0) &   1.0 (18.9) & 4e$+$00 (9.3) &   3.7 (14.0) \\
	$10^{7}$ & 8 & 1e$+$00 (9.2) &   0.9 (16.0) & 8e$-$01 (6.7) &   0.5 (18.9) & 3e$+$00 (9.9) &   2.3 (14.0) \\
	$10^{7}$ & 16 & 6e$-$01 (9.6) &   0.5 (16.0) & 4e$-$01 (7.5) &   0.2 (18.9) & 2e$+$00 (10.0) &   1.5 (14.0) \\
	$10^{7}$ & 24 & 5e$-$01 (10.0) &   0.4 (16.0) & 3e$-$01 (7.8) &   0.2 (18.9) & 1e$+$00 (10.0) &   1.2 (14.0) \\
	$10^{7}$ & 48 & 4e$-$01 (10.0) &   0.3 (16.0) & 2e$-$01 (8.3) &   0.1 (18.9) & 1e$+$00 (10.8) &   0.9 (14.0) \\
	\hline
	$10^{8}$ & 1 & 8e$+$01 (8.8) &  11.8 (17.6) & 7e$+$01 (4.8) &   9.0 (21.1) & 9e$+$01 (9.9) &  14.4 (16.0) \\
	$10^{8}$ & 2 & 4e$+$01 (8.9) &   5.8 (17.6) & 3e$+$01 (5.5) &   3.8 (21.1) & 5e$+$01 (10.0) &   8.2 (16.0) \\
	$10^{8}$ & 4 & 2e$+$01 (9.3) &   3.4 (17.6) & 2e$+$01 (6.2) &   2.3 (21.1) & 3e$+$01 (10.7) &   4.8 (16.0) \\
	$10^{8}$ & 8 & 1e$+$01 (9.9) &   1.9 (17.6) & 1e$+$01 (6.8) &   1.2 (21.1) & 2e$+$01 (11.0) &   2.8 (16.0) \\
	$10^{8}$ & 16 & 8e$+$00 (10.0) &   1.1 (17.6) & 6e$+$00 (7.7) &   0.7 (21.1) & 1e$+$01 (11.0) &   1.7 (16.0) \\
	$10^{8}$ & 24 & 7e$+$00 (10.1) &   1.0 (17.6) & 5e$+$00 (8.0) &   0.7 (21.1) & 9e$+$00 (11.2) &   1.4 (16.0) \\
	$10^{8}$ & 48 & 6e$+$00 (10.9) &   0.9 (17.6) & 5e$+$00 (8.7) &   0.7 (21.1) & 8e$+$00 (11.9) &   1.3 (16.0) \\
	\hline
	$10^{9}$ & 1 & 7e$+$02 (9.3) &  12.6 (19.0) & 7e$+$02 (4.9) &   9.8 (23.0) & 8e$+$02 (11.0) &  13.3 (18.0) \\
	$10^{9}$ & 2 & 3e$+$02 (9.7) &   5.7 (19.0) & 3e$+$02 (5.8) &   4.1 (23.0) & 3e$+$02 (11.0) &   6.1 (18.0) \\
	$10^{9}$ & 4 & 2e$+$02 (10.0) &   3.5 (19.0) & 2e$+$02 (6.6) &   2.6 (23.0) & 2e$+$02 (11.8) &   3.9 (18.0) \\
	$10^{9}$ & 8 & 1e$+$02 (10.1) &   1.9 (19.0) & 1e$+$02 (7.3) &   1.3 (23.0) & 1e$+$02 (12.0) &   2.2 (18.0) \\
	$10^{9}$ & 16 & 7e$+$01 (10.9) &   1.2 (19.0) & 6e$+$01 (7.8) &   0.9 (23.0) & 8e$+$01 (12.1) &   1.4 (18.0) \\
	$10^{9}$ & 24 & 7e$+$01 (11.0) &   1.1 (19.0) & 6e$+$01 (8.3) &   0.9 (23.0) & 7e$+$01 (12.6) &   1.3 (18.0) \\
	$10^{9}$ & 48 & 7e$+$01 (11.0) &   1.1 (19.0) & 6e$+$01 (8.8) &   0.9 (23.0) & 7e$+$01 (13.0) &   1.2 (18.0) \\
	\hline
\end{tabular*}
\caption{Comparison between CPU and GPU versions of Algorithm
\ref{alg:purenewtonsimplex} implemented in \texttt{NewtonCQK.jl} using 64-bit
floating-point. Times are in milliseconds, and were taken as the median over
the 20 generated instances. The values between parentheses are the mean of
iterations.}
\label{tab:simplexgpu}
\end{table}

\subsubsection{Basis pursuit}
\label{sec:basispursuit}

Let $A \in \mathbb{R}^{m \times n}$ be a matrix and a threshold $r > 0$. The
basis pursuit denoising problem~\cite{Chen2001}, or
LASSO~\cite{Tibshirani1996}, has the following alternative constrained
formulation:
\begin{equation}\label{l1ballprojbp}
    \min_x \ \frac{1}{2} \|Ax-b\|^2 \quad
    \text{s.t.} \quad
    \|x\|_1 \leq r.
\end{equation}
As in Section~\ref{sec:svm}, we employ the SPG method to solve this problem.
Each SPG iteration~$k$ requires computing a projection onto an $\ell_1$-ball,
defined by the subproblem
\begin{equation}\label{l1ballprojsubprob}
    \min_x \ \frac{1}{2} \|x - [x^k - \alpha_k A^t (Ax^k - b)]\|^2 \quad
    \text{s.t.} \quad
    \|x\|_1 \leq r,
\end{equation}
where $\alpha_k$ is the spectral step size, see~\eqref{svm_graditer}. The
algorithm is initialized at the origin, $x^0 = 0$. The iterations terminate
when the infinity norm of the projected gradient falls below the tolerance
$10^{-4}$.

Initially, we select two matrices from the collection compiled in Table~3 from
\citet{Lopes2019}: $SC_{1}$ ($m=32{,}768$; $n=65{,}536$) and $SC_{11}$ ($m=800$;
$n=88{,}119$), along with their provided right-hand sides $b$. We set the radius
$r$ to 9.413 and 34, respectively. These are the smallest integer values for
which SPG returns a point whose residual of the linear system $Ax = b$ has
Euclidean norm smaller than $10^{-2}$. Such a choice leads to a sparse
solution for the $SC_{11}$ instance, but not for $SC_{1}$. This difference
unveils distinct behaviors of the warm-start strategy, as discussed below. Our
analysis focuses on the performance gains achieved by using $x^k$ as the
initial guess (warm start) in the specialized
Algorithm~\ref{alg:purenewtonCQK} for solving the projection subproblem, as
well as the potential benefits of parallelization.

Figure~\ref{fig:basispursuitspeedup} illustrates the relative performance of
the sequential algorithms, using the \texttt{NewtonCQK.jl} implementation
without warm start as the baseline. As in previous cases, we report results
for the first and last 100 SPG iterations. For a more comprehensive
comparison, we include the Julia implementation of Condat's method by Dai and
Chen; this allows for a more integrated and consistent comparison within our
SPG framework, which is developed in the same language. To ensure a fair
evaluation, the \texttt{NewtonCQK.jl} implementation is configured to return
sparse vectors, matching the default output of the Dai and Chen code.

\begin{figure}[ht]
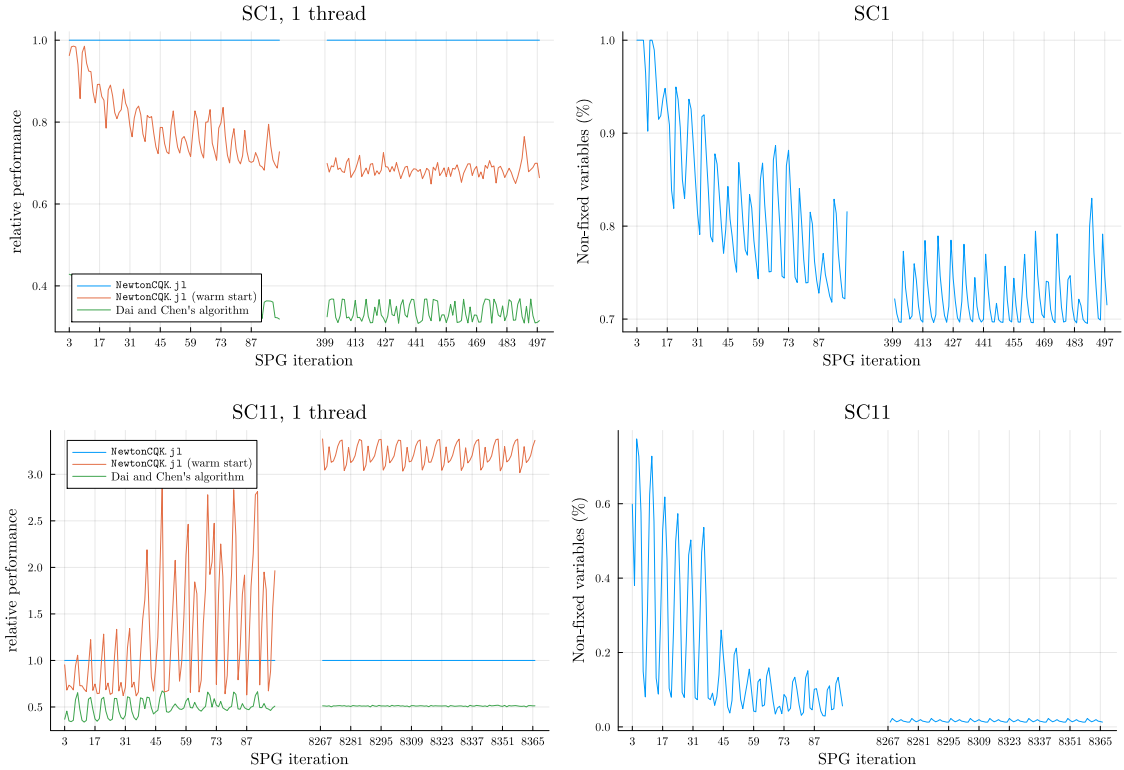

\centering
\def\scale{0.49}
\foreach \problem in {SC1.mat,SC11.mat,SC13.mat,SC20.mat}{
\subfigure{\includegraphics[width=\scale\textwidth]{bp_\problem_time_th1.pdf}}
\subfigure{\includegraphics[width=\scale\textwidth]{bp_\problem_nfixed_th1.pdf}}
}
\caption{On the left, relative performance of CPU algorithms that compute
sparse projections onto the $\ell_1$-ball along the SPG iterations.
\texttt{NewtonCQK.jl} implementation of the specialized
Algorithm~\ref{alg:purenewtonCQK} with no warm start is the base (horizontal
line). On the right, the percentage of non-fixed variables of SPG iterates.}
\label{fig:basispursuitspeedup}
\end{figure}

The \texttt{NewtonCQK.jl} implementations outperform Dai and Chen's code
in all scenarios, which confirms the results presented in
Section~\ref{sec:simplex}. Indeed, the projection onto the $\ell_1$-ball
involves computing a projection onto the simplex as a sub-procedure.
Finally, we note that the implementation by Dai and Chen does not
support a warm-start strategy.

We observe that the warm start is beneficial only for the $SC_{11}$ case but
not for $SC_1$. In $SC_1$, the solution is almost fully dense and the sparsity
structure seems to greatly vary, even in the final iterations. In this case,
using the previous iterate to guess the initial multiplier is not effective,
as the free-variable set does not stabilize. On the other hand, for $SC_{11}$
the solution is sparse and the variation of the free variable set is smaller,
leading to big gains at the final projections.
To confirm this phenomenon, we further consider the two largest instances
from \citet[Table~3]{Lopes2019}, namely, $SC_{13}$ ($m=19{,}996$;
$n=1{,}355{,}191$) and $SC_{20}$ ($m=186{,}414$; $n=685{,}568$). For these
instances, we choose the largest $r$ such that the solution computed by SPG
has at most 1.1\% nonzero components ($r=12{,}053$ for $SC_{13}$ and $r=129$
for $SC_{20}$). Thus, we focus on obtaining sparse solutions, similarly to
\citet{Lopes2019}. Figure~\ref{fig:basispursuitspeedup} indicates that
warm-starting becomes beneficial when the SPG iterates have fewer than
approximately 12.2\% nonzero elements (nonfixed variables).

The performance of the \texttt{NewtonCQK.jl} parallel implementation is shown
in Figure~\ref{fig:basispursuitspeedupparallel}. As observed in
Section~\ref{sec:svm}, the performance gain from parallelism is consistent
throughout the SPG execution.
Notably, higher speedups are observed for the $SC_{11}$, $SC_{13}$,
and $SC_{20}$ instances, for which the iterates are highly sparse. In
particular, for the largest instances $SC_{13}$ and $SC_{20}$, speedups
ranging from 10 to 25 were achieved using 24 or more threads.

\begin{figure}[ht]
\centering
\def\scale{0.49}
\foreach \matrix in {SC1,SC11,SC13,SC20}{
\subfigure{\includegraphics[width=\scale\textwidth]{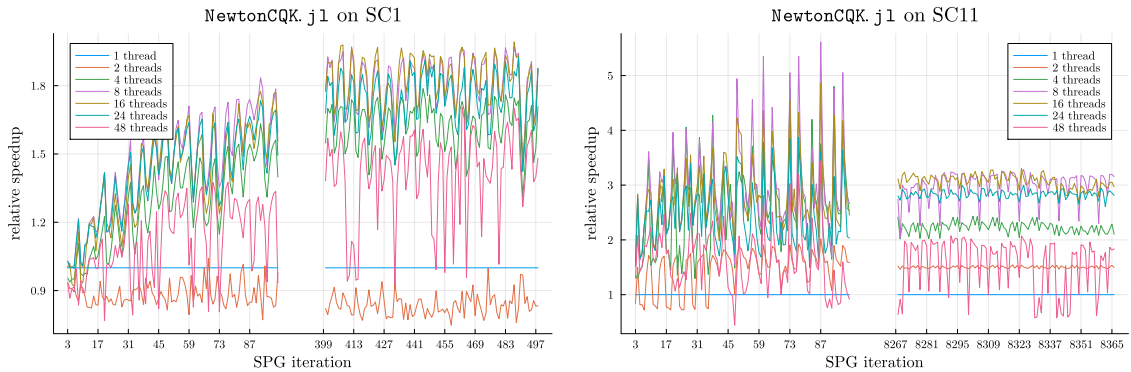}}
}
\caption{Speedups of the parallel implementation of the specialized
Algorithm~\ref{alg:purenewtonCQK} from \texttt{NewtonCQK.jl} for
solving~\eqref{l1ballprojsubprob} (returning sparse solutions) in relation to
the sequential variant along the SPG iterations. All algorithms were executed
using the current SPG iterate as warm start.}
\label{fig:basispursuitspeedupparallel}
\end{figure}

\section{Conclusions}
\label{sec:conclusions}

In this paper, we revisited the semismooth Newton method proposed
in \citet{Cominetti2014} for solving the continuous quadratic knapsack
problem~\eqref{CQK}, a fundamental model that arises in various optimization
contexts. We discussed how this Newton method can be efficiently parallelized,
with a particular focus on projections onto a simplex and an $\ell_1$-ball. Our
open-source Julia package, \texttt{NewtonCQK.jl}, provides high-performance
CPU and GPU implementations that are both flexible and easy to integrate.
Extensive numerical experiments demonstrate that our implementation
outperforms state-of-the-art algorithms, including the widely adopted
variable-fixing method of \citet{Condat2016}. Furthermore, we showcased
the scalability of our approach for large-scale problems; notably, our GPU
implementation achieves speedups of 80--345 times compared to its sequential
CPU counterpart for problems with one million variables or more.

\section*{Declaration of competing interest}

The authors declare that they have no known competing financial interests or
personal relationships that could have appeared to influence the work reported
in this paper.

\section*{Declaration of generative AI and AI-assisted technologies in the writing process}

During the preparation of this work, the authors used Anthropic's Claude Opus
5.5 to assist with the literature overview, the structure of the text, and
language editing, as well as with the development of the MapReduce
implementation based on \texttt{Polyester.jl} and of the index-based variant of
Condat's code discussed in Section~\ref{sec:simplex}. After using this tool,
the authors reviewed and edited the content as needed and take full
responsibility for the content of the publication.

\section*{Acknowledgments}
We would like to thank the referees for carefully reading the original
manuscript and suggesting many improvements that were incorporated in the
text.

\FUNDING 
This research used computational resources from Centro Nacional de
Processamento de Alto Desempenho em São Paulo (CENAPAD-SP), Unicamp / FINEP -
MCTI.

\clearpage
\printbibliography

\end{document}